\documentclass[10pt,letterpaper]{article}
\usepackage[T1]{fontenc}
\usepackage[utf8]{inputenc}
\usepackage{lmodern}

\usepackage{cite}

\usepackage{amsmath}
\usepackage{amsfonts}
\usepackage{amssymb}
\usepackage{graphicx}
\usepackage{mathrsfs}
\usepackage{upref,amsthm,amsxtra,exscale}
\usepackage{dsfont}
\usepackage[colorlinks=true,urlcolor=blue,
citecolor=red,linkcolor=blue,linktocpage,pdfpagelabels,
bookmarksnumbered,bookmarksopen]{hyperref}
\usepackage[cm]{fullpage}
\usepackage[dvipsnames]{xcolor}

\newtheorem{theorem}{Theorem}[section]
\newtheorem{corollary}[theorem]{Corollary}
\newtheorem{remark}[theorem]{Remark}

\newtheorem{lemma}[theorem]{Lemma}
\newtheorem{proposition}[theorem]{Proposition}

\newtheorem{question}[theorem]{Question}

\numberwithin{equation}{section}

\usepackage{booktabs}
\usepackage{multirow} 

\def\r{\mathbb{R}}
\def\rn{\mathbb{R}^N}

\def\s1{\mathbb{S}^1}
\def\n{\mathbb{N}}

\def\eps{\varepsilon}
\def\rh{\rightharpoonup}
\def\io{\int_{\Omega}}
\def\irn{\int_{\r^N}}
\def\vp{\varphi}

\def\vr{\varrho}
\def\o{\Omega}
\def\t{\Theta}
\def\bf{\boldsymbol}

\def\cC{C}

\def\cN{\mathcal{N}}

\def\cS{\mathcal{S}}

\def\bar{\overline}
\def\what{\widehat}
\def\d{\,\mathrm{d}}

\def\i{\mathrm{i}}

\def\ssm{\smallsetminus}

\def\ex{2^{\ast}}

\definecolor{DarkGreen}{rgb}{0.0, 0.5, 0.0}

\author{Mónica Clapp, Cristian Morales-Encinos,\ Alberto Saldaña\footnote{A. Saldaña is supported by SECIHTI grant CBF2023-2024-116 (Mexico) and by UNAM-DGAPA-PAPIIT grant IN102925 (Mexico).}, and Mayra Soares.}
\title{On the asymptotically linear problem for an elliptic equation with an indefinite nonlinearity}
\date{\today}

\begin{document}

\maketitle

\begin{abstract}

We study the semilinear elliptic problem
\[
-\Delta u = Q_\Omega |u|^{p-2}u \quad \text{in } \mathbb{R}^N,
\]
where $Q_\Omega = \chi_\Omega - \chi_{\mathbb{R}^N \setminus \Omega}$ for a bounded smooth domain $\Omega \subset \mathbb{R}^N$, $N \geq 3$, and $1 < p < 2^*$. This equation arises in the study of optical waveguides and exhibits an indefinite nonlinearity due to the sign-changing weight $Q_\Omega$. We prove that, for $p > 2$ sufficiently close to $2$, positive solutions are nondegenerate and the problem admits a unique least energy solution. Our approach combines a detailed analysis of an associated eigenvalue problem involving $Q_\Omega$ with variational methods and blow-up techniques in the asymptotically linear regime. We also provide a comprehensive study of the spectral properties of the corresponding linear problem, including the existence and qualitative behavior of eigenfunctions, sharp decay estimates, and symmetry results. In particular, we establish analogues of the Faber–Krahn and Hong–Krahn–Szegö inequalities in this setting.

\medskip

(primary) 
35J20   
35B09   
35B40 
(secondary).

\medskip

\end{abstract}

\section{Introduction}
Let $\o$ be a bounded smooth open subset of $\rn$ (not necessarily connected), $N\geq 3$, and let 
\begin{equation}\label{Q}
Q_\o(x):=
\begin{cases}
1 & \text{if \ }x\in\o, \\
-1& \text{if \ }x\in\rn\smallsetminus\o.
\end{cases}
\end{equation}	

In this paper we study the uniqueness of minimizers and nondegeneracy of positive solutions to the following semilinear problem
\begin{align}\label{main}
\begin{cases}
    -\Delta u=Q_\o |u|^{p-2}u,\\
    u \in X^p := D^{1,2}(\rn)\cap L^{p}(\rn),
\end{cases}
\end{align}
where $1<p<2^{*},$ $2^{*}=\frac{2N}{N-2}$ is the critical Sobolev exponent, and $D^{1,2}(\rn)=:\{u\in L^{2^*}(\rn)\;:\;\nabla u\in L^2(\rn)\}$. We endow $X^p$ with the norm
\begin{align*}
   \|u\|_{X^p}:=(\|u\|^2+|u|_p^2)^\frac{1}{2}, \qquad\text{where \ }\|u\|:=\left(\irn |\nabla u|^2\right)^\frac{1}{2}\quad \text{ and } \quad |u|_p:=\left(\irn |u|^p\right)^\frac{1}{p}.
\end{align*}

Problem \eqref{main} arises in connection with some models of optical waveguides propagating through a stratified dielectric medium. For a more detailed explanation of the physical relevance of the model, we refer to 
\cite{stuart1991self,stuart1993guidance,ackermann2013concentration,fang2020limiting,clapp2025schrodinger} and the references therein.

It is known that uniqueness does not hold in general for positive solutions to \eqref{main}. In particular, in \cite{clapp2024multiple} the authors show that there exists $\eps>0$ such that, if $p\in(2^*-\eps,2^*)$, then multiple positive solutions to \eqref{main} can be constructed using a Lyapunov-Schmidt procedure, whenever $\Omega$ has nontrivial topology (if it is an annulus, for instance), or if $\Omega$ has a special geometry (\emph{e.g.} if it is a dumbbell type domain), or if $\Omega$ is disconnected (the union of two disjoint balls, for example). The critical case $p=2^*$ is studied in \cite{cfs25}, where Lusternik-Schnirelmann theory is used to obtain multiplicity of positive solutions when a sufficiently thin tubular neighborhood of a manifold with rich topology is removed from $\Omega$.

In this paper, we show that the picture is completely different when considering powers $p$ far from the critical exponent and sufficiently close to 2. This is the first result regarding the uniqueness of minimizers (i.e. least energy solutions) and the nondegeneracy of positive solutions to \eqref{main}.

\begin{theorem}\label{exit p small}
Let $\Omega$ be a bounded smooth open subset of $\rn$ and $N\geq 3$. There exists $p_0>2$ such that the problem \eqref{main} has a unique positive least energy solution for $p\in (2,p_0)$. Moreover, all positive solutions of \eqref{main} are non-degenerate for $p\in (2,p_0)$.
\end{theorem}

The exact notion of nondegeneracy used in Theorem~\ref{exit p small} can be found in Section~\ref{u:sec} and relates to the fact that the linearized problem at a positive solution has only the trivial solution. The definition of least energy solution can be found in Subsection~\ref{sec:minimizers}

Theorem  \ref{exit p small} holds for general bounded smooth sets $\Omega$. This contrasts with the case when the exponent $p$ is far from 2, where symmetries on $\Omega$ can be used to obtain multiplicity results. To illustrate this point, we show a symmetry breaking result for least energy solutions in the slightly subcritical regime, namely, if $p$ is close enough to the Sobolev critical exponent $2^*=\frac{2N}{N-1}$. By using rotations, this yields the existence of infinitely many different least energy solutions. Recall that foliated Schwarz symmetric functions are axially symmetric functions with some monotonicity properties in the polar angle, we refer to Section \ref{eigen:sec} for a precise definition. 

\begin{theorem}\label{thm:mult}
Let $\Omega := \{x\in \rn\::\: \frac{1}{2}<|x|<1\}$ and let $p_k\subset(2,\ex)$ be a sequence such that $p_k\to 2^*$.  Then, passing to a subsequence, \eqref{main} with $p=p_k$ has a nonradial foliated Schwarz least energy solution. 
\end{theorem}

We include in Remark \ref{rmk:table} a table that highlights some differences among positive solutions of \eqref{main} as the exponent \(p\) varies. In particular, we examine the regimes where \(p\) approaches the critical exponent, where \(p\) tends to \(2\) from above and from below, and we also discuss the sublinear range \(p \in (1,2)\), which has been recently analyzed in \cite{CSS26}.

The idea behind the proof of Theorem~\ref{exit p small} is that, in the asymptotically linear case (namely, as $p\to 2$), the solutions inherit the properties of the linear problem ($p=2$), where typically one expects a well-behaved situation.  Hence, to prove Theorem~\ref{exit p small} we need two ingredients: a good knowledge of the linear problem and a good control of the solutions of \eqref{main} as $p$ tends to 2. 

To understand the linear problem, let us consider the eigenvalue problem
\begin{align}\label{P_lambdaI}
\begin{cases}
-\Delta u = \Lambda Q_\o  u,\\
u \in H^{1}(\rn), \quad  u\neq 0.
\end{cases}
\end{align}
A solution to this problem is a pair $(\Lambda, u)\in\r\times H^1(\rn)$ that satisfies \eqref{P_lambdaI} weakly, i.e., $u\neq 0$ and
\begin{align*}
\irn\nabla u\cdot\nabla\vp=\Lambda\irn Q_\o u\vp\qquad\text{for all \ }\vp\in\cC^\infty_c(\rn).    
\end{align*}
This implies, in particular, that
\begin{align*}
    0<\irn|\nabla u|^2=\Lambda\irn Q_\o u^2.
\end{align*}
Therefore, $\Lambda \neq 0$, and $\int_{\mathbb{R}^N} Q_\Omega u^2 > 0$ if and only if $\Lambda > 0$. The existence of solutions with $\Lambda > 0$ was previously established in \cite{sw} in a more general framework. However, to the best of our knowledge, a detailed qualitative analysis of these solutions has not been carried out before. In this paper, we provide an ad hoc proof of the existence of solutions to \eqref{P_lambdaI} and develop new results concerning their structure, including regularity, decay, nodal properties, and symmetry.

To state our main existence result for \eqref{P_lambdaI}, consider the set 
\begin{align}\label{ac}
\cS_\o:=\Big\{v\in D^{1,2}(\rn):\irn Q_\o v^2=1\Big\}.    
\end{align}

The following result shows that the linear problem \eqref{P_lambdaI} resembles the case of the Dirichlet Laplacian in bounded domains. 

\begin{theorem} \label{thm:existence of eigenvalues}
There is a set $\{\phi_k\in\cS_\o:k\in\n\}$ of functions that are orthogonal in $D^{1,2}(\rn)$ such that
$$\|\phi_k\|^2=\inf_{v\in V_k\cap\cS_\o}\|v\|^2=:\Lambda_k,$$
where $V_1:=D^{1,2}(\rn)$ and $V_k$ is the orthogonal complement in $D^{1,2}(\rn)$ of the space generated by the eigenfunctions $\{\phi_1,\ldots,\phi_{k-1}\}$ if $k\geq 2$. Furthermore, 
\begin{itemize}
\item[$(i)$]the pair $(\Lambda_k,\phi_k)$ solves \eqref{P_lambdaI}, 
\item[$(ii)$] $\phi_k\in W^{2,s}_{loc}(\rn)\cap\cC^{1,\alpha}_{loc}(\rn)$ for all $s\in[1,\infty)$, $\alpha\in(0,1)$ and $k\in\n$,
\item[$(iii)$] $|\phi_1|>0$ in $\rn$,
\item[$(iv)$] $0<\Lambda_1<\Lambda_2\leq\cdots\leq\Lambda_k\leq\cdots$ and $\Lambda_k\to\infty$. Moreover, every eigenspace
$$E_k:=\{v\in H^1(\rn):(\Lambda_k,v)\text{ solves }\eqref{P_lambdaI}\}\cup\{0\}$$
has finite dimension and $\dim(E_1)=1$.
\end{itemize}
\end{theorem}

Notice that Theorem~\ref{thm:existence of eigenvalues} does not state that the sequence $(\phi_k)$ generates $D^{1,2}(\rn)$. This is, in fact, \emph{false}, due to the admissibility condition $\int_{\rn}Q_\Omega \phi_k^2=1$ for all $k\in \mathbb N$, see Remark~\ref{no base}.  The main challenge in the proof of Theorem~\ref{thm:existence of eigenvalues} is to  control the inherent lack of compactness of equations in unbounded domains, where \eqref{ac} and the boundedness of the $\Omega$ play a crucial role (see \cite{CSS} a for related problem with $\Omega$ unbounded). 

Next, we explore how the semilinear problem \eqref{main} relates to the linear problem \eqref{P_lambdaI}. For this, we follow the strategy in \cite{st22} and introduce an auxiliary nonlinear eigenvalue problem, which is slightly easier to handle variationally. Let
\begin{align*}
\cS_{\Omega,p}:=\left\{v\in X^p: \irn Q_\o(x)|v|^p=1\right\},
\end{align*}
and consider the minimization problem
\begin{align}\label{alpha_1I}
	\alpha_p :=\inf_{v \in \cS_{\o,p}
    }\| v\|^2.
\end{align}

\begin{theorem}\label{Thm:asym}
    For $p\in (1,2^{*})\smallsetminus\{2\}$, there exists a positive function $v_p\in \cS_{\o,p}$ achieving the infimum in \eqref{alpha_1I}. Moreover, $u_p:=\alpha_p^\frac{1}{p-2}v_p$ is a positive solution to \eqref{main}. Furthermore,
    \begin{enumerate}
        \item  if $\Lambda_1=1$, then 
        \begin{align}
            u_p\to \left( e^{-\frac{1}{2}\irn Q_\o \phi_1^2\ln(\phi_1^2)}\right)\phi_1\in D^{1,2}(\rn)\quad\text{as}\;\;p\to2,
        \end{align}
        where $\phi_1$ is the positive first eigenfunction of \eqref{P_lambdaI} with $\irn Q_\o \phi_1^2 = 1$,
        \item if $\Lambda_1>1,$ then $\lim\limits_{p\nearrow 2}|u_p|_{\infty}=0$ and $\lim\limits_{p\searrow 2}u_p(x)=\infty$ for a.e. $x\in \rn,$ where $|\cdot|_{\infty}$ is the norm in $L^\infty(\rn)$, and
        \item if $\Lambda_1<1,$ then $\lim\limits_{p\nearrow 2}u_p(x)=\infty$ for a.e. $x\in \rn$ and $\lim\limits_{p\searrow 2}|u_p|_{\infty}=0.$ 
    \end{enumerate} 
\end{theorem}

Theorems~\ref{thm:existence of eigenvalues} and~\ref{Thm:asym} are the two main ingredients to prove Theorem~\ref{exit p small} together with suitable blow-up arguments inspired by  \cite{damascelli1999qualitative,dancer2003real,lin1994uniqueness,dis23}, where blow-up techniques have been used to show uniqueness (and nondegeneracy) of solutions in the asymptotically linear regime for the Laplacian and the fractional Laplacian in bounded domains. Nevertheless, important obstacles appear due to the sign-changing weight $Q_\Omega$, and the classical arguments need to be substantially adapted.  In particular, due to these complications, we were not able to show the uniqueness of \emph{all} positive solutions, but instead we only show the uniqueness of minimizers and leave the uniqueness of positive solutions as an open problem.

\begin{question} \label{q1}
Is it true that there exists $p_0>2$ such that the problem \eqref{main} has a unique positive solution for $p\in (2,p_0)$?
\end{question}
See also Remark \ref{obs:rmk} for some comments on Question \ref{q1}.

We also establish the following sharp decay estimates for solutions to \eqref{main}, whenever $p\in (2,p_s]$, where $p_s:=\frac{2N-2}{N-2}$ is sometimes called the Serrin exponent. 
\begin{proposition}\label{decayz}
Let $p\in(2,p_s)$ and let $w\in X^p$ be a positive solution to $-\Delta w = Q_\Omega w^{p-1}$ in $\r^N$. There is $C>1$ such that,
\begin{enumerate}
\item if $p\in(2,p_s)$,
\begin{align}\label{bds}
C^{-1}|x|^{-\frac{2}{p-2}}\leq w(x)\leq C|x|^{-\frac{2}{p-2}}\qquad \text{ for all }|x|>1;
\end{align}
\item if $p=p_s=\frac{2N-2}{N-2}$, 
\begin{align}\label{bds3}
C^{-1}\left(|x|\sqrt{\ln(|x|)}\right)^{2-N}\leq w(x)\leq C\left(|x|\sqrt{\ln(|x|)}\right)^{2-N}\qquad \text{ for all }|x|>1.
\end{align}
\end{enumerate}
\end{proposition}
These sharp estimates answer an open question in \cite[Problem 1.6]{CHS25}. 

\medskip
 
We complement our results with some further qualitative properties of the eigenfunctions. We obtain information on their nodal regions, sharp decay, symmetries, Faber-Krahn and Hong–Krahn–Szegö type inequalities, and other results.  To be more precise, in the following we use $\phi_k$ and $\Lambda_k$ to denote the $k$-th eigenfunction and eigenvalue given by Theorem~\ref{thm:existence of eigenvalues}.  In the next result we write explicitly the domain dependency of the eigenvalues, namely, $\Lambda_k=\Lambda_k(\Omega)$ is the $k$-th eigenvalue associated to \eqref{P_lambdaI}. We also use $\Omega^*$ to denote the Schwarz symmetrization of a bounded open set $\Omega$, namely, the unique ball centered at the origin with the same measure as $\Omega$.

\begin{theorem}\label{theonethm}
Let $\o\subset \rn$ be an open bounded set. 
\begin{enumerate}
    \item (Courant's result) For every $k\geq 2$, $\phi_k$ changes sign and has at most $k$ nodal domains.  
    \item (Sharp decay) For each $k\in \mathbb N$ there is $C>0$ such that 
\begin{align*}
|\phi_k|\leq C |x|^{-\frac{N-1}{2}}e^{-\sqrt{\Lambda_k}|x|} \qquad \text{for all \ }|x|\geq 1.
\end{align*}
This estimate is sharp, in the sense that there is $c>0$ such that 
\begin{align*}
|\phi_1|>c |x|^{-\frac{N-1}{2}}e^{-\sqrt{\Lambda_1}|x|}\qquad \text{for all \ }|x|\geq 1. 
\end{align*}
\item (Faber-Krahn inequality) 
\begin{align*}
\Lambda_1(\Omega)\geq\Lambda_1(\Omega^{*}).    
\end{align*}
\item (Hong–Krahn–Szegö inequality) \begin{align}\label{ine}
        \Lambda_2(\o)>\Lambda_1(2^{-N}\o^*).
    \end{align} 
    This bound is sharp, in the sense that, for every $c>0$,
    \begin{align*}
         \inf\{\Lambda_2(U)\,:\,U\subset\rn\;\text{open bounded such that }|U|=c\,\}=\Lambda_1(B_r(0)),\qquad r:=\left(\frac{c}{2|B_1(0)|}\right)^\frac{1}{N},
     \end{align*}
     and a minimizing sequence is given by $\o_n:=B_r(-ne_1)\cup B_r(ne_1)$ for $n\in \mathbb N,$ where $e_1:=(1,0,\ldots,0)\in \rn$.   
     \item (Symmetries) If $\Omega$ is a radially symmetric open bounded set, then $\phi_1$ is radially symmetric and $\phi_2$ is foliated Schwarz symmetric.
\end{enumerate}
\end{theorem}

Part 1 is an analogue of the famous \emph{Courant nodal domain theorem} \cite{ch}. Here, as usual, the connected components of $\rn\smallsetminus\phi_k^{-1}(0)$ are called the \emph{nodal domains of} $\phi_k$. We emphasize that the fact that each $\phi_k$ changes  sign for every $k\geq 2$ is not an easy consequence of orthogonality (as in the bounded domain case with Dirichlet boundary conditions), because one must consider at all times the (sign-changing) weight $Q_\Omega$; instead, the proof is a consequence of Picone's identity, see Lemma~\ref{schlem}. 

The proof of Part 2 is based on estimates for the Green function of the operator $-\Delta+1$ in $\rn$ adapting the ideas from \cite{MR634248} and on comparison principles. We mention that the same estimates for $\phi_1$ appeared in \cite[equation (1.9)]{ferreri2024asymptotic}, where problem \eqref{P_lambdaI} naturally appears in the context of shape optimization. 

Part 3 is a Faber-Krahn type inequality, this result is shown using symmetrizations and states that the first eigenvalue is minimized at balls (among open sets of fixed measure). In contrast, Part 4 proves that the \emph{second eigenvalue} does not achieve its infimum among open bounded sets of fixed measure.  This is sometimes referred to as the Hong–Krahn–Szegö inequality. Note that the inequality in \eqref{ine} is \emph{strict}, in contrast to what happens in the Dirichlet bounded domain case (where the equality is achieved).

Part 5 shows that the eigenfunctions inherit the symmetries of the domain.  The radial symmetry of $\phi_1$ is an easy consequence of the simplicity of the first eigenvalue, whereas the foliated Schwarz symmetry of $\phi_2$ is shown using polarizations. 

\medskip

To close this introduction, let us mention that we have only considered \emph{positive} eigenvalues of \eqref{P_lambdaI}. However, since the function $Q_\Omega$ is sign-changing one cannot rule out the existence of negative eigenvalues. The following is a nonexistence result of eigenfunctions in $H^1(\rn)$ associated to negative eigenvalues. 

\begin{theorem}\label{thm:poslam}
Let $\Omega \subset \rn$ be a smooth, open and bounded domain and let $(\Lambda,u)$ be a nontrivial solution to \eqref{P_lambdaI}. If $\Omega$ is starshaped (or the disjoint union of starshaped sets), then $\Lambda>0$.
\end{theorem}

The proof relies on a Poho\v zaev-type identity shown in \cite{cfs25}.  The question remains open whether this is true for every $\o$ (not necessarily starshaped).

\begin{question}
Is it true that, for any open bounded subset $\o\subset \rn$, there are no solutions $(\Lambda, u)$ to problem \eqref{P_lambdaI} with $\Lambda<0$?
\end{question}

The paper is organized as follows.  In Section~\ref{sec:eigen} we study the existence of solutions to the eigenvalue problem \eqref{P_lambdaI} and show the qualitative properties stated in Theorem~\ref{theonethm} and the nonexistence Theorem~\ref{thm:poslam}.  Section~\ref{sec:nonlinear} is devoted to the asymptotic analysis of the nonlinear problem as $p\to 2$. In Subsection~\ref{u:sec}, we prove our main result on uniqueness and nondegeneracy in the almost linear regime, that is, Theorem~\ref{exit p small}. In Section~\ref{decay:sec} we derive sharp decay estimates for the solutions to \eqref{main} for small values of $p$. Finally, Section \ref{mult:sec} contains a multiplicity result of least energy solutions in the slightly subcritical regime. 

\section{The eigenvalue problem}\label{sec:eigen}

In this section we study the eigenvalue problem \eqref{P_lambdaI}, where $\o$ is an open bounded subset of $\rn$ and $Q_\o:\rn\to\r$ is given by \eqref{Q}. A solution to this problem is a pair $(\Lambda, u)\in\r\times H^1(\rn)$ such that $u\neq 0$ and
$$\irn\nabla u\cdot\nabla\vp=\Lambda\irn Q_\o u\vp\qquad\text{for all \ }\vp\in\cC^\infty_c(\rn).$$

\subsection{Existence of positive eigenvalues}

Let $\Lambda>0$. Since $(\Lambda,tu)$ with $t\in\r\smallsetminus\{0\}$ solves \eqref{P_lambdaI} whenever $(\Lambda,u)$ does, it suffices to look for solutions $u$ in 
\begin{align}\label{S:def}
\cS_\o:=\Big\{v\in D^{1,2}(\rn):\irn Q_\o v^2=1\Big\}.    
\end{align}
Note that, if $v\in\cS_\o$, then $v\in H^1(\rn)$, because $D^{1,2}(\rn)\subset L^2_{loc}(\rn)$ and, as $\o$ is bounded,
\begin{equation}\label{eq:in H^1}
\int_{\rn\smallsetminus\o}v^2=\io v^2-1<\infty.
\end{equation}
We write $\langle u,v\rangle:=\irn\nabla u\cdot\nabla v$ for the inner product and norm in $D^{1,2}(\rn)$. 

\begin{lemma} \label{lem:existence}
Let $V$ be a closed linear subspace of $D^{1,2}(\rn)$ such that $V\cap\cS_\o\neq\emptyset$. Then,
$$\Lambda_V:=\inf_{v\in V\cap\cS_\o}\|v\|^2$$
is attained at some function $\phi_V\in V\cap\cS_\o$. As a consequence, $\phi_V\in H^1(\rn)$, and it satisfies
\begin{equation}\label{eq:V}
\langle \phi_V, v\rangle=\Lambda_V\irn Q_\o \phi_V v,\qquad\text{for all \ }v\in V\cap H^1(\rn).
\end{equation}
\end{lemma}

\begin{proof}
Let $u_n\in V\cap\cS_\o$ be such that $\|u_n\|^2\to\Lambda_V$. Then $(u_n)$ is bounded in $D^{1,2}(\rn)$ and, passing to a subsequence, $u_n\rh u$ weakly in $D^{1,2}(\rn)$, $u_n\to u$ in $L^2_{loc}(\rn)$ and $u_n\to u$ a.e. in $\rn$. As $V$ is a closed linear subspace of $D^{1,2}(\rn)$ we have that $u\in V$. Furthermore, as $u_n\in\cS_\o$, using Fatou's lemma we obtain
\begin{equation} \label{eq:>1}
1 = \lim_{n\to\infty}\io u_n^2 - \lim_{n\to\infty}\int_{\rn\smallsetminus\o} u_n^2 \leq \io u^2-\int_{\rn\smallsetminus\o} u^2=\irn Q_\o u^2.
\end{equation}
Therefore, $u\neq 0$ and, as a consequence,
\begin{equation*}
0<\|u\|^2\leq\lim_{n\to\infty}\|u_n\|^2=\Lambda_V\leq\left\|\frac{u}{\sqrt{\irn Q_\o u^2}} \right\|^2=\frac{\|u\|^2}{\irn Q_\o u^2}.
\end{equation*}
This implies that $\irn Q_\o u^2\leq 1$ and, together with \eqref{eq:>1} yields $u\in V\cap\cS_\o$ and $\|u\|^2=\Lambda_V$. As shown in \eqref{eq:in H^1}, $u\in H^1(\rn)$.

To prove that $\phi_V:=u$ satisfies \eqref{eq:V} we fix $v\in V\cap H^1(\rn)$. Since $\{w\in H^1(\rn):\irn Q_\o w^2>0\}$ is open in $H^1(\rn)$, there exists $\eps>0$ such that $\irn Q_\o(u+tv)^2>0$ for all $t\in(-\eps,\eps)$. Then, 
$$\frac{u+tv}{\sqrt{\irn Q_\o(u+tv)^2}}\in V\cap\cS_\o\quad\text{for all \ }t\in(-\eps,\eps),$$
and $0$ is a minimum of the function
$$f(t):=\left\|\frac{u+tv}{\sqrt{\irn Q_\o(u+tv)^2}} \right\|^2=\frac{\|u+tv\|^2}{\irn Q_\o(u+tv)^2},\qquad t\in(-\eps,\eps).$$
This function is differentiable and
$$0=f'(0)=2\left(\langle u, v\rangle - \Lambda_V\irn Q_\o u v\right).$$
This completes the proof.
\end{proof}

\begin{lemma} \label{lem:unbounded}
If $\{\phi_k\in\cS_\o:k\in\n\}$ is a set of functions that are orthogonal in $D^{1,2}(\rn)$, $\Lambda_k:=\|\phi_k\|^2$ and $(\Lambda_k,\phi_k)$ solves \eqref{P_lambdaI}, then $(\Lambda_k)$ is unbounded. As a consequence, for each $\Lambda>0$, the space 
$$E_\Lambda:=\{u\in H^1(\rn):(\Lambda, u) \text{ solves }\eqref{P_lambdaI}\}\cup\{0\}$$ 
has finite dimension.  
\end{lemma}

\begin{proof}
As $\{\Lambda_k^{-\frac{1}{2}} \phi_k:k\in\n\}$ is an orthonormal set in $D^{1,2}(\rn)$, passing to a subsequence we have that $\Lambda_k^{-\frac{1}{2}} \phi_k\rh 0$ weakly in $D^{1,2}(\rn)$ and $\Lambda_k^{-\frac{1}{2}} \phi_k\to 0$ strongly in $L^2_{loc}(\rn)$. Therefore,
$$0\leq\Lambda_k^{-1}=\Lambda_k^{-2}\|\phi_k\|^2=\Lambda_k^{-1}\irn Q_\o \phi_k^2\leq\io(\Lambda_k^{-\frac{1}{2}} \phi_k)^2\longrightarrow 0.$$
This proves that $(\Lambda_k)$ is unbounded.

If $E_\Lambda$ had infinite dimension, it would contain a set $\{\phi_k\in\cS_\o:k\in\n\}$ of orthogonal functions in $D^{1,2}(\rn)$ with $\Lambda=\|\phi_k\|^2$. Arguing as before, this would imply that $\Lambda=0$, a contradiction. 
\end{proof}

We are ready to prove Theorem~\ref{thm:existence of eigenvalues}.

\begin{proof}[Proof of Theorem~\ref{thm:existence of eigenvalues}]
We apply Lemma~\ref{lem:existence} inductively to the spaces $V_k$, where $V_1:=D^{1,2}(\rn)$ and $V_k$ is the orthogonal complement in $D^{1,2}(\rn)$ of the space generated by $\{\phi_1,\ldots,\phi_{k-1}\}$ if $k\geq 2$. Lemma~\ref{lem:existence} states that there exists $\phi_k\in V_k\cap\cS_\o$ such that $\|\phi_k\|^2=\Lambda_k$ and 
\begin{equation}\label{eq:solution1}
\langle \phi_k, v\rangle=\Lambda_k\irn Q_\o \phi_k v\qquad\text{for all \ }v\in V_k\cap H^1(\rn).
\end{equation}
Then, $\langle \phi_k,\phi_i\rangle=0$ for all $i=1,\ldots, k-1$ and, as 
$0=\langle \phi_i,\phi_k\rangle=\Lambda_i\irn Q_\o \phi_i\phi_k$
and $\Lambda_i>0$ we get that
\begin{equation}\label{eq:solution2}
\langle \phi_k,\phi_i\rangle=0=\Lambda_k\irn Q_\o \phi_k\phi_i\qquad\text{for all \ }i=1,\ldots,k-1.
\end{equation}

$(i):$ \ \eqref{eq:solution1} and \eqref{eq:solution2} imply that $(\Lambda_k,\phi_k)$ solves \eqref{P_lambdaI}.

$(ii):$ \ This follows from standard regularity arguments, see for instance \cite[Lemma 3.3]{CHS25}.

$(iii):$ \ Replacing $\phi_1$ with $-\phi_1$ if necessary, we may assume that $\phi_1^+:=\max\{\phi_1,0\}\neq 0$. Testing equation \eqref{P_lambdaI} with $\phi_1^+$ we get
$$0<\|\phi_1^+\|^2 = \Lambda_1 \irn Q_\o(\phi_1^+)^2.$$
Then,
$$\what \phi_1^+:=\frac{\phi_1^+}{\sqrt{\irn Q_\o(\phi_1^+)^2}}\in\cS_\o\qquad\text{and}\qquad\Lambda_1=\inf_{v\in\cS_\o}\|v\|^2=\|\what \phi_1^+\|^2.$$
It follows from Lemma~\ref{lem:existence} that $(\Lambda_1,\what \phi_1^+)$ solves \eqref{P_lambdaI}. The strong maximum principle for strong solutions \cite[Theorem~9.6]{gt} applied to the equation 
$$-\Delta\what \phi_1^+ + \Lambda_1\mathds 1_{\rn \ssm \Omega}\what \phi_1^+= \Lambda_1 \mathds 1_{\Omega}\what \phi_1^+\geq 0\qquad\text{in  \ } \rn,$$
where $\mathds 1_A$ is the characteristic function of $A$, yields $\what \phi_1^+>0$ in $\rn$. Hence, $\phi_1=\phi_1^+>0$  in $\rn$.

$(iv):$ \ As $V_1\supset V_2\supset\cdots\supset V_k\supset\cdots$ we have that $0<\Lambda_1\leq\Lambda_2\leq\cdots\leq\Lambda_k\leq\cdots$ and, by Lemma~\ref{lem:unbounded}, $\Lambda_k\to\infty$ and $\dim(E_k)<\infty$. If $\Lambda_1=\Lambda_2$ there would exist two functions $v_1,v_2\in\cS_\o$, orthogonal in $L^2(\rn)$, such that $(\Lambda_1,v_1)$ and $(\Lambda_1,v_2)$ solve \eqref{P_lambdaI}. But then, by $(iii)$, their product $v_1v_2$ would be either strictly positive or strictly negative in $\rn$, contradicting $\irn v_1v_2=0$.
\end{proof}

\begin{remark}\label{no base}
    The set of eigenfunctions $\{\phi_k\in \mathcal{S}_\o:\;k\in\mathbb{N}\,\}$ does not generate the space $D^{1,2}(\rn)$.  Indeed, let $\varphi\in C_c^{\infty}(\rn)$ be such that $\irn Q_\o\varphi^2=-1$ and suppose, by contradiction, that $\varphi_m:=\sum_{i=1}^{m}a_i\phi_i\to\varphi$ in $D^{1,2}(\rn)$ for some $a_i\in\r$. Passing to a subsequence,
    \begin{align}\label{ll2conv}
    \varphi_m\to\varphi\qquad \text{ strongly in $L^2(\o)$ and a.e. in $\rn$.}
    \end{align} Now, since $\irn Q_\o\phi_i\phi_j=0$ if $i\neq j$ and $\irn Q_\o\phi_i^2=1$,
    \begin{align*}
        \irn Q_\o\varphi_m^{2}=\sum_{i=1}^{m}a_i^2\irn Q_\o\phi_i^2=\sum_{i=1}^{m}a_i^2\geq0\qquad \text{ for every }m\in \mathbb N.
    \end{align*}
Then, using Fatou’s Lemma and \eqref{ll2conv},
    \begin{align*}
        0\leq \liminf_{m\to\infty}\irn Q_\o\varphi_m^{2}\leq \int_\o\varphi^2-\int_{\rn\setminus\o}\varphi^2=\irn Q_\o\varphi^2=-1,
    \end{align*}
    a contradiction. 
\end{remark}

 \subsection{Properties of the eigenfunctions}\label{prop:sec}

We show next that the eigenfunctions $\phi_k$ given by Theorem~\ref{thm:existence of eigenvalues} share some of the well known properties of Dirichlet eigenfunctions in bounded domains. Recall that \( \phi_k \) denotes the \( k \)-th eigenfunction associated to the $k$-th eigenvalue $\Lambda_k$ normalized so that 
\begin{align}\label{normalization}
\int_{\rn}Q_\Omega \phi_k^2=1.    
\end{align}

\subsubsection{Monotonicity and scalings} We begin with some easy-to-check properties regarding scalings and monotonicity.
\begin{remark}\label{scaling}
The following properties follow directly from the definition and a change of variables.
\begin{enumerate}
    \item  If $\o_1,\o_2\subset \rn$ are open bounded subsets and $\o_1\subset \o_2$, then $\Lambda_1(\o_1)\geq \Lambda_1(\o_2)$.
    \item  If $t>0$, then $\Lambda_k(t\o)=t^{-2}\Lambda_k(\o)$.
\end{enumerate}
\end{remark} 

\subsubsection{Decay}
Next, we show the sharp decay of $\phi_k$.
\begin{theorem}\label{decay:thm}
    For each $k\in \mathbb N$ there exists $C>0$ such that 
\begin{align}\label{fe}
|\phi_k|\leq C |x|^{-\frac{N-1}{2}}e^{-\sqrt{\Lambda_k}|x|} \qquad \text{for all \ }|x|\geq 1.
\end{align}
This estimate is sharp, in the sense that there exists $c>0$ such that 
\begin{align*}
|\phi_1|>c |x|^{-\frac{N-1}{2}}e^{-\sqrt{\Lambda_1}|x|}\qquad \text{for all \ }|x|\geq 1. 
\end{align*}
\end{theorem}
\begin{proof}
Let $R>0$ be such that $\Omega \subset B_R(0)$. Then,
$$\Delta\phi_k(x)=-\Lambda_k Q_\o(x)\phi_k(x)=\Lambda_k \phi_k(x)\quad\text{for all \ }x\in\rn\smallsetminus B_R(0).$$
Setting $h(r):=r^{-\frac{N-1}{2}}e^{-\sqrt{\Lambda_k}r}$ and $u(x):=h(|x|)$ for $|x|=r>0$, we see that
\begin{align*}
-\Delta u(x) + \Lambda_k u(x)=-h''(r)-\frac{N-1}{r}h'(r)+\Lambda_k h(r)=\frac{(N-3) (N-1)}{4}  e^{-\sqrt{\Lambda_k} r} r^{-\frac{N+3}{2}}\geq 0.
\end{align*}
Fix $a>0$ such that $w(x):=a u(x)\pm \phi_k(x)>0$ on $\partial B_R(0).$  As $-\Delta w + \Lambda_k w\geq 0$ in $\rn\smallsetminus B_R(0)$, we have, by the maximum principle,  that $w>0$ in $\rn\smallsetminus B_R(0)$, which yields \eqref{fe}.

For the sharp estimate, assume first that $\Lambda_1=1.$ The Green function for the operator $-\Delta+1$ in $\rn$ is given by $G(x,y)=|x-y|^{-\frac{N-2}{2}}K_{\frac{N-2}{2}}(|x-y|)$, where $K_\nu$ is the modified Bessel function of order $\nu$; see \cite[Section 4]{MR634248}.  Since $-\Delta \phi_1 + \phi_1 = 2\phi_1 \mathds 1_\Omega$ in $\rn$, where $\mathds 1_\Omega$ is the characteristic function of $\o$, we have that 
\begin{align*}
\phi_1(x)=2\int_\Omega G(x,y)\phi_1(y)\d y\geq 2|\Omega|\left(\inf_\Omega \phi_1\right) \inf_{y\in \Omega}G(x,y).
\end{align*}
The claim for $\Lambda_1=1$ now follows from the asymptotic expansions for $K_{\frac{N-2}{2}}$ given in \cite[Appendix C]{MR634248}. The estimate for general $\Lambda_1>0$ follows by rescaling. 
\end{proof}

\subsubsection{Nodal domains}

Next we give a bound for the number of nodal domains of $\phi_k$ (the connected components of $\rn\smallsetminus\phi_k^{-1}(0)$).

\begin{proposition}\label{nd:prop}  
\( \phi_k \) has at most \(k\) nodal domains.  
\end{proposition}

\begin{proof}
Assume that $\phi_k$ has $m$ nodal domains and that $m\geq k\geq 2$. Choose $k$ nodal domains $U_1,\ldots,U_k$ and define
\[
w_i(x) := 
\begin{cases} 
\phi_k(x) & \text{if } x \in U_i, \\ 
0 & \text{if } x \in \rn \smallsetminus U_i.
\end{cases}
\]
Then, $w_i\in D^{1,2}(\rn)$,
\begin{equation} \label{eq:nodal1}
0<\|w_i\|^2=\langle \phi_k,w_i\rangle = \Lambda_k\irn Q_\o\phi_k w_i=\Lambda_k\irn Q_\o w_i^2\qquad\text{for each \ }i=1,\ldots,k,
\end{equation}
and $\langle w_i,w_j\rangle=0$ if $i\neq j$. Hence, the linear subspace generated by $w_1,\ldots,w_k$ has dimension $k$ and, as a consequence, there exists $w=c_1w_1+\cdots+c_kw_k\neq 0$, $c_i\in\r$, such that 
\begin{equation*}
\langle w,\phi_j\rangle =0\qquad\text{for every \ }j=1,\ldots, k-1.
\end{equation*}
It follows from \eqref{eq:nodal1} that
$$\irn Q_\o w^2 = \sum_{i=1}^k\int_{U_i}Q_\o w^2= \sum_{i=1}^k c_i^2\irn Q_\o w^2_i>0.$$
Setting $t:=\Big(\irn Q_\o w^2\Big)^{-1/2}$ and $u:=tw$ we have that
$$u\in V_k\cap\cS_\o,$$
where $V_k$ is the orthogonal complement in $D^{1,2}(\rn)$ of $\mathrm{span}\{\phi_1,\ldots,\phi_{k-1}\}$, and from \eqref{eq:nodal1} we get
$$\|u\|^2=\sum_{i=1}^k\|tc_i w_i\|^2=\sum_{i=1}^k\Lambda_k\irn Q_\o(tc_i w_i)^2=\Lambda_k \irn Q_\o u^2.$$
This implies that $u$ solves the eigenvalue problem \eqref{P_lambdaI} with $\Lambda=\Lambda_k$.

If $m>k$, then $A:=\rn\smallsetminus (U_1\cup\cdots\cup U_k)$ has nonempty interior. As $u$ vanishes in $A$, the unique continuation principle \cite[Theorem 1.2]{gl} leads to a contradiction. Therefore, $m\leq k$, as claimed.
\end{proof}

\begin{lemma}\label{schlem}
$\phi_k$ changes sign for every $k\geq 2$.
\end{lemma}
\begin{proof}
We use Picone's identity which says that, for any $u,v\in \cC^{1}(\rn)$ with $v>0$ and $u\geq 0$,
\begin{align}\label{picone}
    |\nabla u|^2-\nabla\left(\frac{u^2}{v}\right)\cdot\nabla v=|\nabla u|^2+\frac{u^2}{v^2}|\nabla v|^2-2\frac{u}{v}\nabla u\cdot\nabla v\geq 0,
\end{align}
(see, for instance, \cite[Ec. (1)]{picone}).

Assume, by contradiction, that $\phi_k\geq0$ for some $k\geq 2$. Since $-\Delta \phi_k+ \Lambda_k\mathds 1_{\rn \ssm \Omega}\phi_k= \Lambda_k\mathds 1_{\Omega}\phi_k\geq 0$ in $\rn$, the strong maximum principle \cite[Theorem 9.6]{gt} yields that $\phi_k>0$ in $\rn$. For $\varepsilon>0$, let $\psi_\varepsilon=\frac{\phi_1^2}{\phi_k+\varepsilon}$. Note that $\psi_\varepsilon\in H^{1}(\rn)$. Indeed, by Theorem~\ref{thm:existence of eigenvalues}, $\psi_\varepsilon\in \cC^{1}(\rn)$. Moreover, $0<\psi_\varepsilon<\frac{\phi_1^2}{\varepsilon}$ and \begin{align*}|\nabla\psi_\varepsilon|=\left|\frac{2\phi_1\nabla\phi_1}{\phi_k+\varepsilon}-\frac{\phi_1^2\nabla \phi_k}{(\phi_k+\varepsilon)^2}\right|\leq 2\varepsilon^{-1}\phi_1|\nabla\phi_1|+\varepsilon^{-2}\phi_1^2|\nabla \phi_k|\in L^2(\rn),
\end{align*}
by Theorem~\ref{decay:thm} and Hölder's inequality. Thus, testing with $\psi_\epsilon$ and using \eqref{picone}, 
\begin{align*}
\Lambda_k\irn Q_\o \phi_k\psi_\varepsilon=
 \irn\nabla \phi_k\cdot\nabla\psi_\varepsilon
 =\irn\nabla(\phi_k+\epsilon)\cdot\nabla\left(\frac{\phi_1^2}{\phi_k+\epsilon}\right)\leq\irn|\nabla \phi_1|^2=\Lambda_1.    
\end{align*}
Since $\psi_\eps \phi_k=\frac{\phi_k}{\phi_k+\epsilon}\phi_1^2\to\phi_1^2$ pointwisely as $\varepsilon\to0$ and $\frac{\phi_k}{\phi_k+\epsilon}\phi_1^2\leq\phi_1^2$ for every $\varepsilon>0$, dominated convergence implies that $\Lambda_k\leq\Lambda_1$,
which contradicts Theorem~\ref{thm:existence of eigenvalues}.
\end{proof}

\subsubsection{Eigenfunctions, eigenvalues, and symmetries}\label{eigen:sec}

We now prove the Faber-Krahn type inequality stated in the introduction.

\begin{theorem}[Faber-Krahn inequality]\label{Faber-Krahn}
For every $\Omega\subset \rn$ open bounded set, it holds that 
\begin{align}\label{Faber}
    \Lambda_1(\Omega)\geq\Lambda_1(\Omega^{*}),
\end{align}
where $\o^*$ is the ball centered at the origin with the same volume as $\o$. Moreover, this inequality is an equality if and only if $\Omega$ is a ball.
\end{theorem}
\begin{proof}
Let $\phi_1$ be the positive eigenfunction with eigenvalue $\Lambda_1$ for the problem \eqref{P_lambdaI} in $\o$ that satisfies
$$\irn Q_\o \phi_1^2=1.$$
Let $\phi_1^*$ be the Schwarz symmetrization of $\phi_1$ (also called symmetric-decreasing rearrangement). As shown in \cite[Section 3.3]{ll},
$$\irn\phi_1^2=\irn(\phi_1^*)^2,$$
$(\mathds 1_\o)^*=\mathds 1_{\o^*}$ and $(\phi_1^2)^*=(\phi_1^*)^2$. Then, using \cite[Theorem 3.4]{ll} we see that
$$\io\phi_1^2\leq\int_{\o^*}(\phi_1^*)^2.$$
Therefore,
\begin{align*}
1=\irn Q_\o\phi_1^2=2\io\phi_1^2-\irn \phi_1^2 \leq 2\int_{\o^*}(\phi_1^*)^2-\irn(\phi_1^*)^2=\irn Q_{\o^*}(\phi_1^*)^2.
\end{align*}
By \cite[Lemma 7.17]{ll} we have that
\begin{align*}
\|\nabla \phi_1^{\ast}\|\leq\|\nabla\phi_1\|.
\end{align*}
As a consequence,
\begin{align*}
\Lambda_1=\irn|\nabla\phi_1|^2\geq\irn|\nabla\phi_1^*|^2\geq\frac{\irn|\nabla\phi_1^{\ast}|^2}{\irn Q_{\o^*}(\phi_1^*)^2}\geq\Lambda_1^{\ast},
\end{align*}
as claimed. 

If $\Lambda_1=\Lambda_1^*$, then $\|\nabla \phi_1^{\ast}\|=\|\nabla\phi_1\|$. As shown in \cite[Theorem 1.1]{bz}, this implies that $\phi_1$ is a translation of $\phi_1^*$, i.e., $\phi_1^*=\phi_1(\,\cdot\, -\zeta)$ for some $\zeta\in\rn$. Furthermore, a change of variable gives
$$\Lambda_1\irn Q_{\o+\zeta}(\phi_1^*)^2=\Lambda_1\irn Q_\o\phi_1^2=\|\nabla\phi_1\|^2=\|\nabla \phi_1^{\ast}\|^2=\Lambda_1^*\irn Q_{\o^*}(\phi_1^*)^2=\Lambda_1\irn Q_{\o^*}(\phi_1^*)^2.$$
This implies that $Q_{\o+\zeta}=Q_{\o^*}$. Hence, $\o+\zeta=\o^*$, i.e., $\o$ is a ball, as claimed.
\end{proof}

If $\o$ is radially symmetric, the first and second eigenfunctions have the expected symmetries. 

\begin{proposition} \label{phi1rs}
If $\o$ is a bounded open radially symmetric set, then $\phi_1$ is radially symmetric.
\end{proposition}

\begin{proof}
For any reflection $\vr$ on a hyperplane $V$ through the origin, the pair $(\Lambda_1,\phi_1\circ\vr)$  is a solution to \eqref{P_lambdaI}. Since $\dim(E_1)=1$, $\phi_1(\vr x)=\phi_1(x)$ for all $x\in V$ and $|\phi_1|>0$ in $\rn$, we have that $\phi_1\circ\vr= \phi_1$. As a consequence, since for any $x,y\in\rn$ such that $x\neq y$ and $|x|=|y|\neq 0$, the reflection on the hyperplane through the origin orthogonal to $x-y$ maps $x$ to $y$, we have that $\phi_1(x)=\phi_1(y)$. This proves that $\phi_1$ is radial.
\end{proof}

Next, we show that the eigenfunctions with eigenvalue $\Lambda_2$ are foliated Schwarz symmetric. Let us recall what this means. 

Let $\mathbb{S}^{N-1}=\lbrace x\in\rn\,:\,|x|=1\rbrace$ be the unit sphere. For each $e\in \mathbb{S}^{N-1}$ we consider the half-space $H(e):=\{x\in\rn\,:\,x\cdot e>0\}$, and we denote by $\phi_e$ the composition of the function $\phi:\rn\to\r$ with the reflection $\sigma_e$ with respect to the hyperplane $\partial H(e)$, that is 
\begin{align*}
    \phi_e:\rn\to\r\quad\text{is given by}\quad\phi_e(x):=\phi(\sigma_e(x)),\quad \text{where}\;\;\; \sigma_e(x)=x-2(x\cdot e)e.
\end{align*}
The polarization $\phi_H:\rn\to\r$ with respect to the half-space $H:=H(e)$ is defined by
\begin{align*}
    \phi_H:=\begin{cases}
        \max\{\phi,\phi_e\}&\text{in}\;\; \overline{H},\\
        \min\{\phi,\phi_e\}&\text{in}\;\; \rn\smallsetminus H.
    \end{cases}
\end{align*}

A function $u\in \cC^0(\o)$ is said to be \emph{foliated Schwarz symmetric with respect to} $e^{\ast}\in\mathbb{S}^{N-1}$ if $u$ is axially symmetric with respect to the axis $\r e^{\ast}:=\{te^{\ast}\,:\,t\in\r\}$ (i.e., it is invariant under rotations around this axis), and nonincreasing in the polar angle $\theta:=\arccos(\frac{x}{|x|}\cdot e^{\ast})\in\,[0,\pi]$.

We use the following characterization of foliated Schwarz symmetry.

\begin{lemma} \label{crit}
Let $\o\subset \rn$ be a radially symmetric open set. There exists $e^* \in \mathbb{S}^{N-1}$ such that $u \in \cC^0(\overline{\o})$ is foliated Schwarz symmetric with respect to $e^*$ if and only if for every $e \in \mathbb{S}^{N-1}$ either
\begin{align*}
u \geqslant u_e \quad \text {in } \o\cap H(e) \qquad \text { or } \qquad u \leqslant u_e \quad \text {in } \o\cap H(e). 
\end{align*}
\end{lemma}

\begin{proof}
The proof is given in \cite{b} (see also \cite[Proposition 3.2]{sw12} and \cite{we}). This proof only considers the case where $\Omega$ is connected, but the general case follows easily since each component satisfies the same inequalities that completely determine the axis of symmetry and angular monotonicity.
\end{proof}

We also use the following simple fact. 

\begin{lemma}\label{ort}
    Let $\varphi$ and $f$ Lebesgue-measurable functions in $\rn$ such that $\varphi$ is radially symmetric and $\irn|\varphi f|<\infty$, then $\irn\varphi f_H=\irn\varphi f.$
\end{lemma}
\begin{proof}
    Let $M=\{x\in\rn\,:\;f(x)=f_H(x)\}$. Notice that $\sigma_e(M)=M$ and $f_H=f\circ\sigma_e$ in $M^{c}$, then \begin{align*}\irn \varphi f_H=\int_M \varphi f+\int_{\rn\smallsetminus M}\varphi (f\circ \sigma_e)=\irn \varphi f\end{align*} as claimed.    
\end{proof}

The following result describes the symmetry of the second eigenfunctions. 

\begin{proposition}\label{phi2fss}
Let $\Omega$ be a radially symmetric open bounded subset of $\rn$. Then any eigenfunction of \eqref{P_lambdaI} with eigenvalue $\Lambda_2$ is foliated Schwarz symmetric.
\end{proposition}

\begin{proof}
Let $\phi_1$ and $\phi$ be eigenfunctions with eigenvalues $\Lambda_1$ and $\Lambda_2$, respectively, such that $ \phi_1,\phi\in\cS_\o$. Notice that $-\Delta\phi_1=\Lambda_1 Q_\Omega\phi_1$ is radially symmetric. By Lemma~\ref{ort},
    \begin{align*}
        \irn \nabla\phi_1\cdot\nabla \phi_H=\irn-(\Delta \phi_1)\phi_H=\irn-(\Delta \phi_1)\phi=\irn\nabla\phi_1\cdot\nabla\phi=0.
    \end{align*}
    Therefore, $\phi_H$ is an eigenfunction of \eqref{P_lambdaI} with eigenvalue $\Lambda_2$. Notice that $|\phi-\phi_e|=2\phi_H-(\phi+\phi_e)$ in $H:=H(e)$ and $-|\phi-\phi_e|=2\phi_H-(\phi+\phi_e)$ in $\rn\smallsetminus H(e)$. By standard elliptic regularity arguments, $\phi,\phi_H\in W_{loc}^{2,s}(\rn)\cap\cC_{loc}^{1,\alpha}(\rn)$ for all $s\in [1,\infty)$ and $\alpha\in(0,1)$. Note that $2\phi_H-(\phi+\phi_e)=0=-(2\phi_H-(\phi+\phi_e))$ on $\partial H(e)$, $|\phi-\phi_e|\in H^{1}(\rn)\cap W_{loc}^{1,s}(\rn)$ for all $s\ge 1$, and
\begin{align}\label{sol max}
-\Delta|\phi-\phi_e|
&=\begin{cases}
\Lambda_2Q_\o[(\phi_H-\phi)+(\phi_H-\phi_e)]&\text{in }\; H(e),\nonumber\\
\Lambda_2Q_\o[(\phi-\phi_H)+(\phi_e-\phi_H)]&\text{in}\; \rn\smallsetminus H(e),
\end{cases}\\
&=\begin{cases}
\;\,\,\Lambda_2Q_\o|\phi-\phi_e|\quad\text{in}\;\; H(e),\\
-\Lambda_2Q_\o|\phi-\phi_e|\quad\text{in}\; \rn\smallsetminus H(e),\\
\end{cases}\nonumber\\
&=\Lambda_2 Q_\o(\mathds 1_{H(e)}-\mathds 1_{\rn\smallsetminus H(e)})|\phi-\phi_e|
    \end{align}
    Since $|\Lambda_2 Q_\o(\mathds 1_{H(e)}-\mathds 1_{\rn\setminus H(e)})|=\Lambda_2\in L^{\infty}(\rn)$, standard elliptic regularity yields that 
    \begin{align}\label{reg}
    w_e:=|\phi-\phi_e|\in W^{2,s}_{loc}(\rn)\cap \cC_{loc}^{1,\alpha}(\rn)\cap \cC^{\infty,\alpha}(H(e)\cap\o)\cap\cC_{loc}^{\infty,\alpha}(H(e)\cap(\rn\smallsetminus\o))
    \end{align}
     for all $s\geq 1,\ \alpha\in(0,1)$.
    Hence, $w_e$ is a strong solution to $\eqref{sol max}$; in particular, 
    \begin{align*}
    -\Delta w_e\geq 0\quad \text{in $H(e)\cap\o$}\qquad \text{ and  }\qquad -\Delta w_e+\Lambda_2 w_e=0\quad \text{ in $H(e)\cap(\rn\smallsetminus\o)$.} 
    \end{align*}
     By the  strong maximum principle, either $w_e\equiv 0$ or $w_e>0$ in each connected component of $H(e)\cap \o$. We claim that, if $w_e>0$ in some connected component $A\subset\Omega$, then $w_e>0$ in $H(e)$. Indeed, let $B$ be the connected component of the positive nodal set $\{x\in \mathbb R^n\::\: w_e>0\}$ which contains $A$ and assume, by contradiction, that there exists $x_0\in \rn$ such that $w_e(x_0)=0$. Then there exists $y_0\in \partial B$ and a small ball $U$ with $y_0\in \partial U$, which is contained
    either in $B\cap \Omega$ or in $B\cap \rn\smallsetminus\Omega$.  In any case, by the Hopf Lemma, $\nabla w_e(y_0)\neq 0.$ However, making a Taylor expansion around $y_0$, in view of \eqref{reg}, it follows that
    \begin{align*}
    0\leq |\phi-\phi_e| = w_e(x) = \nabla w_e (y_0)\cdot(x-y_0)+o(|x-y_0|) \mbox{ as } |x-y_0|\to 0.
    \end{align*} 
    But this implies that $\nabla w_e(y_0)=0$, which yields a contradiction.  Thus, $w_e>0$ in $H(e)$ as claimed. This shows that either $\phi\geq \phi_e$ or $\phi\leq\phi_e$ in $H(e)$ and the claim follows from Lemma~\ref{crit}.     
\end{proof}

Recall that $\phi_2$ denotes an eigenfunction related to the second eigenvalue $\Lambda_2:=\Lambda_2(\o)$ such that $\irn Q_\o\phi_2=1$. Let
\begin{align*}
    A_{+}:=\{x\in\rn\,:\,\phi_2(x)>0\}\qquad 
    \text{and}\qquad 
    A_-:=\{x\in\rn\,:\,\phi_2(x)<0\}.
\end{align*}
By Lemma~\ref{schlem}, $A_{\pm}\neq\emptyset$. The following result is one of the main ingredients for our Hong–Krahn–Szegö-type inequality.
\begin{lemma}\label{MIN SECONG EIGENVALUE}
Let $\o_{\pm}:=\o\cap A_{\pm}$ and let $\Lambda_1(\o_+)$ and $\Lambda_1(\o_-)$ denote the  first eigenvalues associated to $\o_+$ and $\o_-$, respectively. Then,
\begin{align}
    \Lambda_2(\o)>\max\{\Lambda_1(\o_+),\Lambda_1(\o_-)\}.
\end{align}
\end{lemma}

\begin{proof}
We show first that $\o_{\pm}\neq\emptyset$. By Lemma~\ref{schlem}, $\phi_2$ changes sign.  Since $        \lim_{|x|\to\infty}|\phi_2(x)|=0$, there are $a,b\in\rn$ such that
    \begin{align*}
        \phi_2(a)=\sup_{x\in \rn}\phi_2(x)>0\quad\text{and}\quad\phi_2(b)=\inf_{x\in\rn}\phi_2(x)<0.
    \end{align*}
    If $a\in\partial\o$, then $\o_+\neq\emptyset$ by continuity.  If $a\in\rn\setminus\partial\o$, then $\Lambda_2 Q_\o(a)\phi_2(a)=-\Delta\phi_2(a)\geq 0$, hence $Q_\o(a)>0$ and, therefore, $a\in\o$, which also yields that $\o_+\neq\emptyset$. Arguing similarly we also have that $\o_-\neq\emptyset$.
    
    Testing with $\phi_2^{+}=\max\{\phi_2,0\}$,
    \begin{align*}
        \|\phi_2^{+}\|^2=\Lambda_2\irn Q_\o\phi_2(\phi_2^+)=\Lambda_2\irn (Q_\o\mathds 1_{A_+})(\phi_2^{+})^2=\Lambda_2\irn Q_{\o_+}(\phi_2^{+})^2,
    \end{align*}
because $Q_\o\mathds1_{A_+}=\mathds 1_{\o\cap A_+}-\mathds1_{(\rn\setminus\o)\cap A_+} $ and $\phi_2^+=0$ in $\rn \setminus A_+$. Then $\phi:=\phi_2^{+}/\sqrt{\irn Q_{\o_{+}}(\phi_2^+)^2}\in S_{\o_+}$ (recall \eqref{S:def}) and 
    \begin{align}\label{l1l2}
        \Lambda_1(\o_+)=\inf_{v\in S_{\o_+}}\|v\|^2\leq \|\phi\|^2=\Lambda_2.
    \end{align}
    Since $\phi\equiv0$ in $A_{-}$, the inequality in \eqref{l1l2} must be strict (by unique continuation). We obtain the same estimate for $\Lambda_2(\o_{-})$ using similar arguments and the claim follows.
\end{proof}

We denote by $B_r$ the open ball of radius $r$ centered at the origin in $\rn$ and by
$$r(c):=\left(\frac{c}{|B_1|}\right)^{1/N}$$
the radius of the ball of volume $c$ in $\rn$. As before,
$$\o^*:=B_{r(|\o|)}.$$

\begin{theorem}\label{thm:hks}
\begin{itemize}
\item[$(i)$] For any open bounded subset $\o$ of $\rn$,
$$\Lambda_2(\o)>\Lambda_1(B_{r(|\o|/2)}).$$
\item[$(ii)$] For any $c>0$,
$$\inf\{\Lambda_2(\o):\o\subset\rn\text{ is open and bounded and }|\o|=c \, \}=\Lambda_1(B_{r(c/2)}).$$
A minimizing sequence is given by $\o_n:=B_{r(c/2)}(-ne_1)\cup B_{r(c/2)}(ne_1)$, where $e_1=(1,0,\ldots,0)\in \rn$. 
\end{itemize}
\medskip

\noindent Therefore, the second eigenvalue does not achieve its infimum among open bounded sets of fixed volume.
\end{theorem}

\begin{proof}
$(i):$ By Theorem~\ref{Faber-Krahn} and Remark~\ref{scaling}, 
\begin{align*}
    \Lambda_1(\o_+)\geq \Lambda_1(\o_+^*)
    =\left(\frac{|\o_+|}{|B_1|}\right)^{2/N}\Lambda_1(B_1).
\end{align*}
Similarly, $\Lambda_1(\o_-)\geq\left(\frac{|\o_-|}{|B_1|}\right)^{2/N}\Lambda_1(B_1).$ Hence, by Lemma~\ref{MIN SECONG EIGENVALUE},
    \begin{align*}
        \Lambda_2(\o)
        &>\max\{\Lambda_1(\o_+),\Lambda_1(\o_-)\}
        \geq \max\left\{\left(\frac{|\o_+|}{|B_1|}\right)^{2/N}\Lambda_1(B_1), \left(\frac{|\o_-|}{|B_1|}\right)^{2/N}\Lambda_1(B_1)\right\}\\
        &=\frac{\Lambda_1(B_1)}{|B_1|^{2/N}}\left(\max\{|\o_+|,|\o_-|\}\right)^{2/N}
        \geq \left(\frac{|\o|/2}{|B_1|}\right)^{2/N}\Lambda_1(B_1)
        =\Lambda_1(B_{r(|\o|/2)}),
    \end{align*}
    where we used that $2\max\{|\o_+|,|\o_-|\}\geq |\o_+|+|\o_-|=|\o|$. 
    
$(ii):$ For simplicity, we take $c:=2|B_1|$. A scaling argument using Remark~\ref{scaling} gives the statement for any $c>0$. Let $\phi_1$ denote the  eigenfunction associated to the first eigenvalue $\Lambda_1:=\Lambda_1(B_1)$ with $\int_{\rn}Q_{B_1} \phi_1^2=1$.  It is clear that $\phi_1(\,\cdot\,-n e_1)$ and $\phi_1(\,\cdot\,+ne_1)$ are the 
      eigenfunctions associated to the first eigenvalues $\Lambda_1(B_1(ne_1))$ and $\Lambda_1(B_1(-n e_1))$, respectively, and 
     $\Lambda_1=\Lambda_1(B_1(n e_1))=\Lambda_1(B_1(-ne_1)).$
     Let $u_n(x):=\phi_1(x+ne_1)-\phi_1(x-ne_1)$ for $x\in \rn$. Using that $\int_{\rn}Q_{B_1} \phi_1^2=1$ and that $\phi_1$ is uniformly bounded in $\rn$ and decays exponentially (by Theorem~\ref{decay:thm}), it is not hard to verify that 
     \begin{align}\label{ochiquita}
         \irn Q_{\o_n} u_n^2=2+o(1)
     \end{align}
and 
\begin{align}\label{de2}
       \irn\nabla \phi_1(x+ne_1)\cdot\nabla\phi_1(x-ne_1)=\Lambda_1(B_1)\irn Q_{B_1(-ne_1)}\phi_1(x+ne_1)\phi_1(x-ne_1)=o(1)
   \end{align}
as $n\to\infty.$ Let $f_n$ be the  eigenfunction associated to the first eigenvalue $\alpha_n:=\Lambda_1(\o_n)$ with $\irn Q_{\o_n}f_n^2=1.$ Since the first eigenvalue is simple,
    \begin{align}
       \phi_1(x)=\phi_1(\sigma(x))\quad\text{and}\quad f_n(x)=f_n(\sigma(x))\quad \text{ for all }x\in \rn,\qquad \text{where \ }\sigma(x):=(-x_1,...,x_N).
   \end{align}
   Then, changing variables,
\begin{align*}
    \irn Q_{\o_n}(x)f_n(x)\phi_1(x-ne_1)\d x
    =\irn Q_{\o_n}(x)f_n(x)\phi_1(\sigma(x)-ne_1)\d x
    =\irn Q_{\o_n}(x)f_n(x)\phi_1(x+ne_1)\d x
\end{align*}
   and, as a consequence, 
   \begin{align*}
       \irn\nabla f_n\cdot\nabla u_n&=\alpha_n\left(\irn Q_{\o_n}(x)f_n(x)\phi_1(x-ne_1)\d x-\irn Q_{\o_n}(x)f_n(x)\phi_1(x+ne_1)\d x\right)=0.
   \end{align*}
   Hence, $f_n$ and $u_n$ are orthogonal in $D^{1,2}(\rn)$. If $       \psi_n:=u_n/\sqrt{\irn Q_{\o_n}u_n^2},$ then, using statement $(i)$, \eqref{ochiquita} and \eqref{de2},
    \begin{align*}
\Lambda_1(B_1)<  \Lambda_2(\o_n)=\inf_{\substack{v\in S_{\o_n}\\ v\perp f_n}}\|v\|^2\leq\|\psi_n\|^2=\frac{2\irn|\nabla\phi_1|^2}{2+o(1)}+o(1)=\Lambda_1(B_1)+o(1)\qquad \text{ as }n\to\infty,
    \end{align*}
as claimed.
\end{proof}

Summarizing, we obtain Theorem~\ref{theonethm}.

\begin{proof}[Proof of Theorem~\ref{theonethm}]

\textbf{Part 1.} This is an immediate consequence of Proposition~\ref{nd:prop} and Lemma~\ref{schlem}, which establish the nodal properties and sign-changing behavior of the eigenfunctions.

\textbf{Part 2.} The sharp decay estimates are proven in Theorem~\ref{decay:thm}.

\textbf{Part 3.} The Faber--Krahn-type inequality is established in Theorem~\ref{Faber-Krahn}.

\textbf{Part 4.} The Hong--Krahn--Szegö-type inequality is proved in Theorem~\ref{thm:hks}.

\textbf{Part 5.} The symmetry properties of the eigenfunctions are derived in Propositions~\ref{phi1rs} and~\ref{phi2fss}, which treat the radial symmetry of the first eigenfunction and the foliated Schwarz symmetry of the second one, respectively.

\end{proof}

\subsection{Nonexistence of negative eigenvalues}\label{sec:neg}

The following Poho\v zaev-type identity was proved in \cite[Theorem 3.1]{cfs25}.

\begin{theorem}\label{thm:poho}
Let $\Theta$ be a (possibly unbounded) smooth open subset of $\rn$ and $f:\r\to\r$ be a continuous function. Set $F(s):=\int_0^s f$. If $u\in \cC^{1}(\rn)\cap \cC^2(\Theta)\cap \cC^2(\rn\smallsetminus\overline{\Theta})\cap D^{1,2}(\rn)$ is a solution to
\begin{equation} \label{eq:problem_pohozhaev}
-\Delta u=Q_\Theta f(u)
\end{equation}
such that $F(u)\in L^1(\rn)$, then $u$ satisfies
\begin{align*}
\frac{1}{2^*} \int_{\rn}|\nabla u|^2-\int_{\rn}Q_\Theta F(u)  + \frac{2}{N}\int_{\partial\t}F(u)\, \zeta\cdot \nu_\Theta \, d\zeta=0,
\end{align*}
where $\nu_\t$ denotes the exterior unit normal to $\t$.
\end{theorem}

We apply the Poho\v zaev identity to prove our nonexistence result, Theorem~\ref{thm:poslam}, stated in the introduction.

\begin{proof}[Proof of Theorem~\ref{thm:poslam}]
Applying Theorem~\ref{thm:poho}, we have that 
\begin{align*}
0
&=\frac{1}{2^*} \int_{\rn}|\nabla u|^2-\frac{\Lambda}{2}\int_{\rn}Q_\Omega u^2  + \frac{\Lambda}{N}\int_{\partial\Omega}u^2 \zeta\cdot \nu_\Omega \, d\zeta=\left(\frac{1}{2^*}-\frac{1}{2}\right)\int_{\rn}|\nabla u|^2  + \frac{\Lambda}{N}\int_{\partial\Omega}u^2\, \zeta\cdot \nu_\Omega \, d\zeta.
\end{align*}
This implies that $\frac{\Lambda}{N}\int_{\partial\Omega}u^2\, \zeta\cdot \nu_\Omega \, d\zeta>0$, and the claim follows.
\end{proof}

\section{The almost linear problem}
\label{sec:nonlinear}

In this section we consider the nonlinear problem \eqref{main} and we study the behavior of least energy solutions for $p$ close to $2$.

\subsection{Asymptotic behavior of ground states}
\label{sec:minimizers}

We adapt the approach followed in \cite{st22}. For $p\in (1,2^*)$ consider the Banach space $X^p:=D^{1,2}(\rn)\cap L^p(\rn)$ equipped with the norm
$$\|u\|_{X^p}:=(\|u\|^2+|u|_p^2)^\frac{1}{2},$$
where $|\cdot|_p$ denotes the norm in $L^p(\rn)$.
Set
\begin{align*}
\cS_{\Omega,p}:=\left\{v\in X^p: \irn Q_\o(x)|v|^p=1\right\},
\end{align*}
and consider the minimization problem
\begin{align}\label{alpha_1}
	\alpha_p :=\inf_{v \in \cS_{\o,p}
    }\| v\|^2.
\end{align}

 \begin{lemma} \label{lem:bounds for alpha}
There are $a_0>0$ and $a_1>0$ such that $0<a_0\leq \alpha_p\leq a_1<\infty$ for all $p\in (1,2^*)$.
\end{lemma}

\begin{proof}
Let $u\in\cS_{\Omega,p}$. Since $\o$ is bounded, using Hölder's and Sobolev's inequalities we obtain
    \begin{align*}
        1&=\irn Q_\o(x)|u|^p\leq\io|u|^p
        \leq|\o|^{\frac{2^*-p}{2^*}}\left(\io|u|^{2^*}\right)^{p/2^*} \leq |\o|^{
\frac{2^*-p}{2^*}}\left(\irn|u|^{2^*}\right)^{p/2^*}
        \leq |\o|^{\frac{2^*-p}{2^*}}C\left(\irn|\nabla u|^{2}\right)^{p/2},
    \end{align*}
    where $C=C(N)>0$. Therefore, $a_0:=\min_{p\in[1,2^*]}(|\o|^{\frac{(p-2^*)2}{2^*p}}C^{-2/p})\leq \|u\|^2$ for every $u\in\cS_{\o,p}$ and $p\in(1,2^*)$.  This shows that $\alpha_p\geq a_0>0$.
    
To prove the other inequality, fix $\rho\in \cC_c^\infty(\o)$ such that $\rho\neq 0$. Then, the function $f:[1,\infty)\to D^{1,2}(\rn)$, given by $f(p):=\frac{\rho}{\left(\irn Q_\Omega|\rho|^p\right)^{1/p}}$, is continuous and $f(p)\in \cS_{\o,p}$. Hence,
\begin{align}\label{bound }
\alpha_p\leq \max_{p\in [1,2^*]}\|f(p)\|^2=:a_1<\infty \qquad\text{for all}\;p\in [1,2^*],
\end{align} 
as claimed.  
\end{proof}

\begin{theorem}
For each $p\in (1,2^*)\backslash\{2\}$, there exists a nonnegative function $v_p\in \cS_{\o,p}$ such that \eqref{alpha_1} is achieved at $v_p$ and $u_p:=\alpha_p^\frac{1}{p-2} v_p$ is a solution to \eqref{main}.
\end{theorem}

\begin{proof}
The proof is similar to that of Lemma ~\ref{lem:existence}. It is given in \cite[Lemma 2.1]{fang2020limiting}. 
\end{proof}

A solution $u_p:=\alpha_p^\frac{1}{p-2} v_p$ to \eqref{main} such that $v_p$ is a minimizer for $\alpha_p$ is called a \emph{least energy solution}. These solutions are the minimizers on the space $X^p$ of the energy functional given by
$$J(u):=\frac{1}{2}\irn|\nabla u|^2-\frac{1}{p}\irn Q_\o|u|^p$$
if $p\in(1,2)$, and they are the minimizers of $J$ on the Nehari manifold $\cN$, defined in \cite[(3.2)]{CHS25}, if $p\in(2,2^*)$.

\begin{lemma}\label{regularity}
    Every solution to \eqref{main} belongs to $W_{loc}^{2,s}(\rn)\cap\cC_{loc}^{1,\alpha}(\rn)\cap\cC_{loc}^{\infty}(\o)\cap\cC_{loc}^{\infty}(\rn\setminus\overline{\Omega})$ for all $s\in[1,\infty)$ and $\alpha\in (0,1)$.
\end{lemma}

\begin{proof}
    See \cite[Lemma 2.4]{CHS25}.
\end{proof}

\begin{theorem}\label{thm:asym1}
For $p\in (1,2^*)$, let $v_p\in \cS_{\o,p}$ be a nonnegative minimizer of \eqref{alpha_1}. Then $v_p\to\phi_1$ in $D^{1,2}(\rn)$ as $p\to 2$, where $\phi_1$ is the positive eigenfunction of \eqref{P_lambdaI} with eigenvalue $\Lambda_1$ such that $\irn Q_\Omega \phi_1^2 =1$. Moreover, 
    \begin{align}\label{asymp}
    \lim_{p\to 2}\left( \frac{\Lambda_1}{\alpha_p}\right)^{\frac{1}{2-p}}
    =e^{-\frac{1}{2}
    \int_{\mathbb R^N} Q_\Omega \phi_1^2 \ln(\phi_1^2)}.    
    \end{align}
\end{theorem}

\begin{proof}
Fix $\delta>0$ so that  $[2-\delta,2+\delta]\subset(1,2^*)$. Let $p_n\in (1,2^*)$ be such that $p_n\to 2$ and $v_n\in S_{\o,p_n}$ be such that $\alpha_{p_n}=\|v_n\|^2$. By Lemma~\ref{lem:bounds for alpha} we have that $(v_n)$ is uniformly bounded in $D^{1,2}(\rn)$. Hence, passing to a subsequence,  $v_n\rh v$ weakly in $D^{1,2}(\rn)$, $v_n\to v$ in $L_{loc}^q(\rn)$ for every $q\in [2-\delta,2+\delta]$, and $v_n\to v$ a.e. in $\rn$. Then, $v\geq 0$. We claim that
    \begin{align}\label{claim}
        \lim_{n\to\infty}\io v_n^{p_n}=\lim_{n\to\infty}\io v_n^2=\io v^2<\infty. 
    \end{align}
For $n$ large enough, we have that $2s+(1-s)p_n\in (2-\delta/2,2+\delta/2)$ for all $s\in (0,1)$.  Since $\lim_{x\to 0^+}x^{\delta/2}\ln(x)=0$ and $\lim_{x\to\infty}\ln(x)/x^{\delta}=0$, there exists $C=C(\delta)>0$ such that $x^{\delta/2}|\ln(x)|\leq C x^{\delta}$ for all $x\in (0,\infty)$,  and
    \begin{align*}
        \left|\io v_n^{p_n}-\io v_n^2\right|&\leq|p_n-2| \io\int_0^1v_n^{2s+(1-s)p_n}|\ln(v_n)| \d s \d x\\
        &\leq|p_n-2| \io(v_n^{2-\delta/2}+v_n^{2+\delta/2})|\ln(v_n)|\d x
        \leq C|p_n-2|\io(v_n^{2}+v_n^{2+\delta})\d x=o(1) 
    \end{align*}
    as $n\to\infty$, which yields \eqref{claim}. Furthermore, as $v_n\in \cS_{\o,p_n}$, Fatou's lemma and \eqref{claim} yield 
\begin{align*}
\int_{\rn\smallsetminus \o}v^2\leq\liminf_{n\to\infty}\int_{\rn\smallsetminus \o}|v_n|^{p_n}=\lim_{n\to\infty}\io|v_n|^{p_n}-1=\io v^2-1.
    \end{align*}
This shows that $v\in L^2(\rn)$ and that $1\leq\irn Q_\o v^2$. In particular, $v\neq 0$. Hence, 
\begin{align}\label{d1}
\liminf_{n\to\infty}\alpha_{p_n}=\liminf_{n\to\infty}\|v_n\|^2\geq\|v\|^2\geq\frac{\|v\|^2}{\irn Q_\o v^2}\geq\Lambda_1.
\end{align}
On the other hand, by Theorem~\ref{decay:thm}  and Lemma~\ref{regularity}, the first positive eigenfunction $\phi_1$ of \eqref{P_lambdaI} with $\irn Q_\Omega\phi_1^2=1$, satisfies $\phi_1\in L^p(\rn)$ for every $p\in [1,\infty)$. Fixing $\alpha\in (0,\sqrt{\Lambda_1})$, there exists $C>0$ such that $\phi_1^{p_n}(x)\leq Ce^{-\alpha |x|}$ for large enough $n$, and the dominated convergence theorem implies that $\lim_{p\to 2}\int_{\rn} Q_\Omega\phi_1^p=\irn Q_\Omega\phi_1^2=1$. Therefore, 
    \begin{align}\label{d2}
\Lambda_1=\|\phi_1\|^2=\lim_{n\to\infty}\frac{\|\phi_1\|^2}{\Big(\irn Q_\o\phi_1^{p_n}\Big)^{2/p_n}} \geq \limsup_{n\to\infty}\alpha_{p_n}.
    \end{align}
     From \eqref{d1} and \eqref{d2} we get $\irn Q_\o v^2=1$,  $\|v\|^2=\Lambda_1$, $v_n\to v$ strongly in $D^{1,2}(\rn)$, and
     \begin{align}\label{converge alpha-p}
         \lim_{n\to\infty}\alpha_{p_n}=\Lambda_1.
     \end{align}
Since $\text{dim}(E_1)= 1$, we have that $v=\phi_1$.

It remains to show \eqref{asymp}. Since $u_n:=\alpha_{p_n}^\frac{1}{p_n-2} v_{n}$ solves \eqref{main}, we have that     
     \begin{align*}
     \irn \Lambda_1Q_\Omega\phi_1v_n=\irn \nabla\phi_1\cdot\nabla v_n=\irn\alpha_{p_n}Q_\o v_n^{p-1}\phi_1.    
     \end{align*}
     Hence,
     \begin{align*}
0&=\irn Q_\Omega\phi_1v_n\left(\frac{\alpha_{p_n}}{\Lambda_1}v_n^{p_n-2}-1\right)=\irn Q_\Omega\phi_1v_n\int_0^1\frac{d}{ds}\Big(\Big(\frac{\alpha_{p_n}}{\Lambda_1}\Big)^{\frac{1}{p_n-2}}v_n\Big)^{s(p_n-2)} \d s \d x\\
         &=(p_n-2)\irn Q_\Omega\phi_1 v_n\int_0^1\ln\Big(\Big(\frac{\alpha_{p_n}}{\Lambda_1}\Big)^{\frac{1}{p_n-2}}v_n\Big)\Big(\Big(\frac{\alpha_{p_n}}{\Lambda_1}\Big)^{\frac{1}{p_n-2}}v_n\Big)^{s(p_n-2)}\d s\d x\\
         &=(p_n-2)\irn Q_\Omega\phi_1v_n\left(\frac{1}{p_n-2}\ln\Big(\frac{\alpha_{p_n}}{\Lambda_1}\Big)+\ln(v_n)\right)\int_0^1\left(\frac{\alpha_{p_n}}{\Lambda_1}\right)^sv_n^{s(p_n-2)}\d s\d x. 
         \end{align*}
Since $v_n\to\phi_1$ in  $D^{1,2}(\rn)$ there exists $h\in L^{2^*}(\rn)$ such that, after passing to a subsequence, $v_n\leq h$ a.e. in $\rn$ (see, for instance, \cite[Lemma A.1]{willem1996minimax}). By \eqref{converge alpha-p}, there exists $M>1$ such that
         \begin{align*}
             \left(\frac{\alpha_{p_n}}{\Lambda_1}\right)^{s}\leq M\quad\text{for all}\;s\in(0,1)\;\text{and}\;n\in\mathbb{N}.
         \end{align*}
          Since $p_n\to2,$ we have that $s(p_n-2)+1\in(1/2,2)$ for $n$ large enough. By Theorem~\ref{decay:thm}, there exists $C>0$ such that $\phi_1\leq Ce^{-\sqrt{\Lambda_1}|x|}$ for all $x\in\rn$. It follows that
         \begin{align*}
0<v_n\phi_1\left(\frac{\alpha_{p_n}}{\Lambda_1}\right)^sv_n^{s(p_n-2)}\leq M\phi_1h^{s(p_n-2)+1}
             \leq MCe^{-\sqrt{\Lambda_1} |x|}(h^{1/2}+h^2).
         \end{align*}
     On the other hand, fixing $\varepsilon\in(0,\frac{1}{2})$ such that $2+\varepsilon<2^{*}$, we have that
     \begin{align*}
         \Big|v_n\phi_1\ln(v_n)\Big(\frac{\alpha_{p_n}}{\Lambda_1}\Big)^sv_n^{s(p_n-2)}\Big|& \leq MC_\varepsilon\phi_1v_n^{s(p_n-2)+1+\varepsilon} \leq MC_\varepsilon\phi_1(h^{\frac{1}{2}+\varepsilon}+h^{2+\varepsilon})\\
         & \leq MC_\varepsilon C e^{-\sqrt{\Lambda_1}|x|}(h^{\frac{1}{2}+\varepsilon}+h^{2+\varepsilon}),
     \end{align*}
     where $C_\varepsilon>0$ is such that $|x\ln(|x|)|\leq C_\varepsilon |x|^{1+\varepsilon}$ for all $x\in\r$. Then, by dominated convergence theorem, 
     \begin{align*}
\irn\int_0^1 Q_\Omega v_n\phi_1\Big(\frac{\alpha_{p_n}}{\Lambda_1}\Big)^sv_n^{s(p_n-2)}\d s\d x>0
     \end{align*}
for $n$ large enough. Moreover,
\begin{align*}
    \frac{\ln\big(\frac{\alpha_{p_n}}{\Lambda_1}\big)}{p_n-2}=-\frac{\irn\int_0^1Q_\Omega v_n\phi_1\ln(v_n)\left(\frac{\alpha_{p_n}}{\Lambda_1}\right)^sv_n^{s(p_n-2)}\d s\d x}{\irn\int_0^1Q_\Omega v_n\phi_1\left(\frac{\alpha_{p_n}}{\Lambda_1}\right)^sv_n^{s(p_n-2)}\d s\d x}
\end{align*}
and, using that $\irn Q_\Omega\phi_1^2=1$,
\begin{align*}
    \ln\Big(\Big(\frac{\alpha_{p_n}}{\Lambda_1}\Big)^{\frac{1}{p_n-2}}\Big)=\frac{\irn Q_\Omega\phi_1^2\ln(\phi_1)}{\irn Q_\Omega\phi_1^2}+o(1)=-\frac{1}{2}\irn Q_\Omega\phi_1^2\ln(\phi_1^2)+ o(1),
\end{align*}
as claimed. 
\end{proof}

Theorem~\ref{Thm:asym} stated in the introduction is now an immediate consequence of the previous result.

\subsection{Uniqueness of minimizers and nondegeneracy of positive solutions}\label{u:sec}

In this section we prove Theorem~\ref{exit p small}. The proof follows the strategy used in \cite{dis23}, but several adaptations are needed in order to control the lack of compactness at spatial infinity.

For $p>2,$ consider the linearized problem at a positive solution $u$ to problem \eqref{main} given by
\begin{equation}
\label{linearized}
-\Delta h= Q_\Omega (p-1)u^{p-2}h\qquad \text{ in } \rn.
\end{equation}
A solution $u$ to \eqref{main} is said to be \emph{nondegenerate} if \eqref{linearized} has only the trivial solution $h=0$, i.e. if $\mu=0$ is not an eigenvalue of the linearized operator $L_u:=-\Delta -Q_\Omega (p-1)u^{p-2}$.  The spectrum and the associated eigenfunctions for $L_u$ can be obtained analogously to Section~\ref{sec:eigen}. 

We first need to prove some auxiliary results.  We begin with a nonexistence result. 

\begin{lemma}\label{nonex:lem}
Let $U\subset \rn$ be an open bounded set such that $\Lambda_1(U)=1.$ If $\Omega\subset \rn$ is an open set such that $U\subset \Omega$, $U\neq \Omega,$ then the problem 
\begin{align}\label{ulin}
-\Delta u = Q_\Omega u\quad \text{ in }\rn,\qquad u>0\quad\text{ in }\rn,
\end{align}
does not admit a bounded positive solution. As a consequence, if $\Omega = \rn_+:=\{x\in \rn\::\: x_N>0\}$ is a halfspace or if $\Omega=\rn$, then \eqref{ulin} does not admit a bounded positive solution. 
\end{lemma}
\begin{proof}
Assume, by contradiction, that $u$ is a bounded positive solution to problem \eqref{ulin}.
By maximum principles, arguing as in Theorem~\ref{decay:thm}, we have that $u$ must have exponential decay and, therefore, $u\in X^p.$ Let $\phi_U$ denote the eigenfunction associated to the first eigenvalue $\Lambda_1(U)=1$. Testing \eqref{ulin} with $\phi_U$,
\begin{align*}
 \int_{\rn}Q_U u\phi_U = \int_{\rn} u (-\Delta \phi_U)=\int_{\rn}(-\Delta u) \phi_U = \int_{\rn}Q_\Omega u\phi_U,
\end{align*}
and, therefore, 
\begin{align}\label{uphieq}
\int_{\rn}(Q_\Omega - Q_U)u\phi_U=0.    
\end{align}
On the other hand, since $U\subset \Omega,$ we have $Q_\Omega - Q_U\geq 0$ in $\rn$ and $Q_\Omega - Q_U=2>0$ in $\Omega\backslash U$, which contradicts \eqref{uphieq}, because $u\phi_U>0$ in $\rn$. 

In particular, if $\Omega = \rn_+$ or $\o=\rn$, then \eqref{ulin} does not admit a positive solution, because $\rn_+$ contains the ball $B_r(re_N)$, where $r>0$ is such that $\Lambda_1(B_r(re_N))=1.$
\end{proof}

\begin{lemma}\label{pclose1}
Let $(p_n)$ be a sequence in $(2,2^*)$ such that $p_n\to 2$, $u_n$ be a positive solution to \eqref{main} with $p=p_n$, and let $M_n:=|u_n|_{\infty}$. Then,
\begin{align*}
M_n^{p_n-2}\to \Lambda_1(\Omega) \qquad \text{and}\qquad \frac{u_n}{M_n}\to \phi_1 \text{ \ in \ }C^{2,\alpha}_{loc}(\Omega)\cap C^{2,\alpha}_{loc}(\rn\backslash\Omega)\cap C^{1,\alpha}_{loc}(\rn)\qquad\text{as \ }n\to\infty  
\end{align*}
for some $\alpha\in(0,1)$, where $\Lambda_1(\Omega)$ and  $\phi_1$ denote, respectively, the first eigenvalue and eigenfunction normalized so that $|\phi_1|_{\infty}=1$.
\end{lemma}

\begin{proof}
{\sl Step 1. We show that $M_n^{p_n-2}$ is bounded.}\\
By contradiction, assume that $M_n^{p_n-2}\to \infty$ and let $x_n\in\overline{\Omega}$ be such that $M_n=u_n(x_n)$ (the fact that $x_n\in \overline{\Omega}$ follows easily from \eqref{main}, since $-\Delta u_n(x_n)\geq 0$ whenever $x_n\not\in \partial \Omega,$ because $x_n$ is a global maximum). Define
\begin{align*}
w_n(y):=\frac{1}{M_n}u_n\left( \mu_ny+x_n\right) 
\end{align*}
where $\mu_n:=M_n^{\frac{2-p_n}{2}}\to 0$ as $n\to\infty$. Then $0 \leq w_n\leq 1$, $w_n(0) = 1$, and
\begin{align}\label{w_n eq}
-\Delta w_n = Q_{\Omega_n} w_n^{p_n-1}=:f_n\quad &\text{ in } \rn,
\end{align}
where $\Omega_n:=\left\{y\in\rn:\, \mu_ny+x_n\in \Omega \right\}$. Up to subsequences,
\begin{align}\label{2cases}
\text{  either }\quad  \operatorname{dist}\left(x_{n},\,\partial\Omega\right) \mu_{n}^{-1} \rightarrow+\infty \quad \text{ or } \quad \operatorname{dist}\left(x_{n},\,\partial\Omega\right) \mu_{n}^{-1} \rightarrow \rho \geq 0.
\end{align}
Assume the first case holds, so that $\Omega_n \rightarrow \rn$ as $n \rightarrow+\infty$. 
For every $R>0$ we claim that there exists $n_R\in\mathbb N$ and $C=C(N,\alpha,R)>0$
such that
\begin{align}
\label{vnlocReg}
w_n\in C^{2,\alpha}(\overline B_R) \ \mbox{ and }\ \|w_n\|_{C^{2,\alpha}(\overline B_R)}\leq C \quad \text{ for all } n\geq n_R
\end{align}
and all $\alpha\in(0,1)$. In order to prove \eqref{vnlocReg} let us fix $R>0$. Then, since $\Omega_n \rightarrow \rn$ as $n \rightarrow+\infty$, there exists $n_R\in\mathbb N$ such that $\overline B_{4R}\subset\Omega_n$ for any $n\geq n_R$.
Since $0\leq w_{n} \leq 1$ in $\rn$, then $|f_n|\leq 1$ in $\rn$. In particular, by elliptic regularity, $w_n\in C^{1,\alpha}(\overline B_{2R})$ for any $\alpha \in (0,1)$ and $\|w_n\|_{C^{1,\alpha}(B_{2R})}\leq C$ for some constant $C=C(N,\alpha,R)>0$. Then $f_n\in C^{\alpha}(B_{2R})$ with uniform $C^{\alpha}$-bound and so \eqref{vnlocReg} follows by elliptic regularity, see for instance \cite[B.1 Theorem]{Struwe}.

By \eqref{vnlocReg}, using the Arzel\`{a}-Ascoli theorem and a diagonal argument we obtain that there exists a function $w\in C_{loc}^{2,\frac{\alpha}{2}}(\rn)$ such that,  passing to a subsequence,  $w_{n} \rightarrow w$ in $C^{2,\frac{\alpha}{2}}_{loc}(\rn)$. Passing 
\eqref{w_n eq} to the limit, we see that $w$ solves $-\Delta w=w$ in $\rn$ pointwisely.
Furthermore, $0\leq w\leq 1$, so $w$ is bounded in $\rn$ and $w>0$ in $\rn$ by the maximum principle. However, by Lemma~\ref{nonex:lem}, no such $w$ can exist and we have reached a contradiction. 

If the second case in \eqref{2cases} holds then we may assume $x_{n} \rightarrow x_{0} \in \partial \Omega .$ With no loss of generality assume that $\nu(x_{0})=-e_{N}$, where $\nu$ denotes the exterior unit normal vector at $\partial \Omega$. Let
\begin{align*}
w_{n}(y):=\frac{u_{n}\left(\mu_{n} y+\xi_{n}\right)}{M_n},
\end{align*}
where $\xi_{n} \in \partial \Omega$ is the closest point to $x_{n}$ on $\partial \Omega$. Let $D_{n}:=\left\{y \in \mathbb{R}^{N}: \mu_{n} y+\xi_{n} \in \Omega\right\}$ and observe that
\begin{align}
\label{0inD}
0 \in \partial D_{n}
\end{align}
and $D_{n} \rightarrow \rn_+:=\left\{y \in \mathbb{R}^{N}: y_{N}>0\right\}$ as $n \rightarrow+\infty$.  It also follows that $w_{n}$ satisfies that
\begin{align}\label{v_n eq2}
-\Delta w_n = Q_{D_n}w_n^{p_n-1}\quad \text{ in } \rn.
\end{align}
Moreover, setting
$$
y_{n}:=\frac{x_{n}-\xi_{n}}{\mu_{n}}\ (\in \overline{D_n})
$$
we have that $\left|y_{n}\right|=\operatorname{dist}\left(x_{n},\,\partial\Omega\right) \mu_{n}^{-1}$ and $w_{n}(y_{n})=1$.  Now, arguing similarly as in the first case, we obtain that $w_{n} \rightarrow w$ in $C^{2,\frac{\alpha}{2}}_{loc}(\rn_{+})\cup C^{2,\frac{\alpha}{2}}_{loc}(-\rn_{+})\cup C^{1,\frac{\alpha}{2}}_{loc}(\rn)$, where $w$ satisfies that $0 \leq w \leq 1$ in $\rn$, $w(y_{0})=1$, and $w$ is a bounded solution to
 \begin{align*}
 -\Delta w=Q_{\rn_+}w \qquad  \text { in } \rn.
 \end{align*}
By the maximum principle, $w>0$ in $\rn$, which contradicts Lemma~\ref{nonex:lem}. This shows that $M_n^{p_n-2}$ is bounded and concludes the proof of {\sl Step 1}.

\medskip

{\sl Step 2. Conclusion.}\\
$M_n^{p_n-2}$ is bounded by {\sl Step 1.} Thus, up to a subsequence, $M_n^{p_n-2}\to \mu$. Let $z_n:=\frac{u_n}{M_n}$, then 
\begin{align*}
-\Delta z_n= Q_{\Omega} M_n^{p_n-2}\,z_n^{p_n-1}=:g_n \quad \text{   in } \rn,\qquad 
 0<z_n\leq 1 \quad \text{   in } \rn.
\end{align*}
Arguing as before, by elliptic regularity, $z_n$ converges to $z$ in $C_{loc}^{2,\frac{\beta}{2}}(\Omega)\cup C_{loc}^{2,\frac{\beta}{2}}(\rn\backslash\Omega)\cup C^{1,\frac{\beta}{2}}(\rn)$; furthermore, $0\leq z\leq 1$ and it  satisfies that
\begin{align}\label{zeq}
-\Delta z=Q_{\Omega}\mu z\quad \text{ in } \rn,\qquad 
z\geq 0\quad \text{ in }\rn,
\end{align}
Let $x_n\in\overline{\Omega}$ be such that $M_n=u_n(x_n)$, then  $1=\lim_{n\to \infty}z_n(x_n)=z(\bar{x})$, for some $\bar{x}\in \overline{\Omega}$. By the maximum principle $z>0$ in $\rn$ and then, using that only the first eigenfunction is positive and its eigenvalue is simple, we have that $\mu=\Lambda_1(\Omega)$, $z=\phi_1$. Note that any bounded solution to \eqref{zeq} must have the decay stated in Theorem~\ref{decay:thm} (by using a comparison argument as in the proof of this lemma) and, therefore, belongs to $H^1(\rn).$  This ends the proof.
\end{proof}

Next, we show a uniform estimate for normalized positive solutions.

\begin{lemma}\label{decayzn:lem}
Fix $q\in(2,2\frac{N-1}{N-2})$. Let $(p_n)$ be a sequence in $(2,q)$ such that $p_n\to 2$, $u_n$ be a positive solution to \eqref{main} with $p=p_n$, $M_n:=|u_n|_{\infty}$, and $z_n:=\frac{u_n}{M_n}$. There are $R>1$ and  $C>0$ such that
\begin{align}\label{cota}
    z_n(x)\leq C|x|^{-\frac{2}{q-2}}\qquad\text{if }|x|\geq R\text{ and for all }n\in \mathbb{N}.
\end{align}
\end{lemma}
\begin{proof}
    Let $f(x)=|x|^{-\frac{2}{q-2}}$. By a simple computation, we have that $\Delta f=c_qf^{q-1}$ in $\rn\backslash\{0\},$ where $c_q=\frac{2(q-(q-2)(N-1))}{(q-2)^2}$ which is positive for $q\in(2,2\frac{N-1}{N-2})$. Thus, $\tau(x):=c_q^{1/q-2}f(x)$ satisfies that $-\Delta \tau=-\tau^{q-1}$ in $\rn\setminus\{0\}.$ Let $R>1$ be such that $\o\subset B_R(0)$ and define $\psi_n:=\tau-z_n$. Then, using that $u_n$ is a positive solution of \eqref{main} with $p=p_n$, we have that
\begin{align*}
    -\Delta\psi_n=-\tau^{q-1}+M_n^{p_n-2}z_n^{p_n-1}\quad\text{in }\rn\setminus B_R(0).
\end{align*}
Let 
\begin{align*}
    a_n(x):=\begin{cases}
        \frac{\tau(x)^{p_n-1}-z_n^{p_n-1}(x)}{\tau(x)-z_n(x)}M_n^{p_n-2},&\text{if } \tau(x)\neq z_n(x),\\
        0,&\text{if }\tau(x)=z_n(x). 
    \end{cases}
\end{align*}
Observe that $a_n\geq 0$ and
\begin{align}\label{acotamientouniforme}
    -\Delta\psi_n+a_n\psi_n=-\tau^{q-1}+M_n^{p_n-2}\tau^{p_n-1}\quad\text{in }\rn\setminus B_R(0).
\end{align}
Since $2<p_n<q$ and $M_n^{p_n-2}\to\Lambda_1(\o)$, note that
\begin{align}\label{acotamientouniforme2}
-\tau^{q-1}+M_n^{p_n-2}\tau^{p_n-1}
\geq \left(\frac{\Lambda_1}{2}-\tau^{q-p_n}\right)\tau^{p_n-1}\geq 0\quad\text{in }\rn\setminus B_R(0)
\end{align}
for $n$ and $R$ sufficiently large. Hence, by \eqref{acotamientouniforme} and \eqref{acotamientouniforme2}, there are $m\in\mathbb{N}$ and $R_0>R$ such that 
\begin{align*}
    -\Delta\psi_n+a_n\psi_n\geq0\qquad\text{in }\rn\setminus B_{R_0}(0)\text{ \ for }n\geq m.
\end{align*}
On the other hand, by Theorem~\ref{decay:thm}, $\phi_1(x)\leq C|x|^{-\frac{N-1}{2}}e^{-\sqrt{\Lambda_1(\o)}|x|}$. Hence, there is $R_1>R_0$ such that $\tau(x)-\phi_1(x)>0$ for all $x\in\rn\setminus \overline{B_{R_1}(0)}$. Since, by Lemma~\ref{pclose1}, $z_n\to\phi$ in $C_{loc}^{1,\alpha}(\rn)$ for some $\alpha\in(0,1)$, there is $\tilde{m}\in\mathbb{N}$ such that
\begin{align}\label{mp2}
\psi_n>0\qquad \text{ on  $\partial B_{R_1}(0)$ \ for all $n\geq\tilde{m}$.    }
\end{align}
 Since $z_n(x)\to 0$ as $|x|\to\infty$ (see Proposition~\ref{decayz}) and \eqref{acotamientouniforme2} and \eqref{mp2} hold, it follows from the maximum principle that $z_n(x)<\tau(x)$ for all $|x|\geq R_1$ and $n \geq \widetilde m$, as claimed.
\end{proof}

We are ready to prove our uniqueness and nondegeneracy result.

\begin{proof}[Proof of Theorem~\ref{exit p small}]
Let $(p_n)$ be a subsequence of $(2,2^*)$ such that $p_n\to 2$, $u_n$ be a solution to \eqref{main} with $p=p_n$, and $M_n:=|u_n|_{\infty}$.

{\sl Step 1. Nondegeneracy of positive solutions:}  Assume, by contradiction, that there exists a non-trivial solution $h_n$ to the linearized problem \eqref{linearized} with $p=p_n>2$ and $p_n\to 2$, namely,
\begin{align}
\label{linearizedproblemFonNondegeneracy}
-\Delta h_n=Q_\Omega (p_n-1)u_n^{p_n-2}h_n\ =:g_n \quad {\text{ in }} \rn.
\end{align}
We may assume that $|h_n|_{\infty}=1$. Using Lemma~\ref{pclose1},
\begin{align*}
|g_n|_{\infty}
\leq 2|u_n|^{p_n-2}_{\infty}
= 2M_n^{p_n-2}\leq C.
\end{align*} 
So, by elliptic regularity, $h_n\to h$ in $C^{1,\alpha}_{loc}(\rn)$. Furthermore, again by elliptic regularity, $h_n  \in C^\infty(\Omega)\cap C^{1}(\rn)\cap L^\infty(\rn)$ (the smoothness follows by a standard bootstrap procedure). Taking $h_n$ as a test function, by Lemma~\ref{pclose1},  we derive that
\begin{align*}
\|h_n\|^2
&=\int_{\rn} Q_\Omega (p_n-1)u_n^{p_n-2}h_n^2
\leq (p_n-1)\int_\Omega u_n^{p_n-2}h_n^2
\leq 2M_n^{p_n-2}|\Omega| |h_n|_{\infty}^2
= 2M_n^{p_n-2}|\Omega|\leq C.
\end{align*}

Hence, up to a subsequence, $h_n$ converges to $h$ weakly in  $D^{1,2}(\rn)$ and strongly in $L_{loc}^2(\rn)$. Letting $n\to\infty$ in the weak formulation of \eqref{linearizedproblemFonNondegeneracy} 
we obtain that $h$ is a weak solution of
\begin{align*}
-\Delta h = Q_\Omega \Lambda_1(\Omega) h \quad {\text{ in }} \rn,
\end{align*}
because, by Lemma~\ref{pclose1}, one has that
\begin{equation}\label{limitOfunp}
(p_n-1)u_n^{p_n-2} = (p_n-1)M_n^{p_n-2}\left(\frac{u_n}{M_n}\right)^{p_n-2}=\Lambda_1(\Omega)+o(1)\quad \text{pointwisely in $\rn$ as }n\to\infty\end{equation}
and
\begin{align}\label{limit2}
|(p_n-1)u_n^{p_n-2}|_{\infty}\leq \Lambda_1(\Omega)+1\qquad \text{ for all }n\in\mathbb N.
\end{align}

Note that both $h_n$ and $\phi_1$ achieve their supremum in $\overline{\Omega}$. Indeed, by (uniform) local regularity estimates and the fact that $h_n\in D^{1,2}(\rn)$, we have that $h_n$ decays uniformly to 0 at infinity. Let $(x_n)$ be the points in $\rn$ where $h_n(x_n)=|h_n|_\infty=1$, and assume, by contradiction, that $x_n\not\in \overline{\Omega}$. Then, $0\leq -\Delta h_n(x_n)= -(p_n-1)u_n(x_n)^{p_n-1}h_n(x_n)<0$, which yields a contradiction. That  $\phi_1$ achieves its supremum in $\overline{\Omega}$ can be shown similarly. 

Then, by the local convergence, $h=\phi_1$, where $\phi_1$ is the  eigenfunction associated to the first eigenvalue $\Lambda_1(\Omega)$ such that $|\phi_1|_\infty=1$. Note that, integrating by parts,
\begin{equation}\label{segno}
0=\int_{\rn} h_n(-\Delta u_n) - u_n(-\Delta h_n)\, dx
= \int_{\rn} Q_\Omega u_n^{p_n-1}h_n\, dx
=(2-p_n)M_n^{p_n-1}\int_{\rn} Q_\Omega z_n^{p_n-1}h_n\, dx.
\end{equation}
Using Lemmas~\ref{pclose1},~\ref{decayzn:lem}, and the fact that $|h_n|_\infty=1$, we can use dominated convergence to obtain that
\begin{align*}
0=\lim_{n\to\infty}\int_{\rn} Q_\Omega z_n^{p_n-1}h_n\, dx
=\int_{\rn} Q_\Omega \phi_1^2\, dx>0.    
\end{align*}
This yields a contradiction and the claim follows. 

\medskip

{\sl Step 2. Uniqueness of minimizers.} \\
Arguing by contradiction, assume that there is a sequence $(p_n)$ in $(2,2^*)$ such that $p_n\to2$ and, for each $n$, there are two different least energy solutions $u_n=\alpha_n^{\frac{1}{p_n-2}} v_n$ and $\hat{u}_n=\alpha_n^{\frac{1}{p_n-2}} \hat v_n$, where $\alpha_n:=\alpha_{p_n}$ and $v_n$ and $\hat v_n$ are different minimizers of $\alpha_n$ (see Section~\ref{sec:minimizers}). By Theorem~\ref{thm:asym1},
\begin{align}\label{vs}
v_n,\hat v_n \to c\phi_1\quad \text{ in  \ }D^{1,2}(\rn),
\end{align}
where $\phi_1$ is the positive first eigenfunction with $|\phi_1|_\infty=1$ and $c:=(\int_{\rn}Q_\Omega \phi_1^2)^{-1/2}$.  By Lemma~\ref{pclose1} we also know that 
\begin{align}\label{vs2}
\frac{v_n}{|v_n|_\infty}=\frac{u_n}{|u_n|_\infty}\to \phi_1\quad \text{ in  \ }C^{1,\alpha}_{loc}(\rn).
\end{align}
Let $\rho>0$ be such that $\int_{B_\rho(0)}|\nabla \phi_1|^2>0$. Then, by \eqref{vs} and \eqref{vs2},
\begin{align*}
0
=\lim_{n\to\infty}\int_{B_\rho(0)}\left|\nabla \left(\frac{v_n}{|v_n|_\infty}-\frac{v_n}{c}\right)\right|^2
=\lim_{n\to\infty}\int_{B_\rho(0)}\left|\nabla v_n\right|^2\left(\frac{1}{|v_n|_\infty}-\frac{1}{c}\right)^2
=c^2\int_{B_\rho(0)}\left|\nabla \phi_1\right|^2\lim_{n\to\infty}\left(\frac{1}{|v_n|_\infty}-\frac{1}{c}\right)^2.
\end{align*}
We conclude that 
\begin{align}\label{Ms}
|v_n|_\infty\to c\quad \text{ and, similarly, }\quad |\hat v_n|_\infty\to c.
\end{align}
Since $v_n\neq \hat v_n$ by assumption, the function $w_n:= \frac{v_n-\hat{v}_n}{|v_n-\hat{v}_n|_{\infty}}$ is a nontrivial solution of
\begin{align}\label{w_neq}
-\Delta w_n=Q_\Omega \alpha_n c_nw_n=:g_n \quad \text{ in } \rn
\end{align}
where, by the mean value theorem,
\begin{align*}
c_n(x):=\int_0^1 (p_n-1)(t v_n(x) +(1-t) \hat{v}_n(x))^{p_n-2} \d t.
\end{align*}
Since $t v_n(x) +(1-t) \hat{v}_n(x)$ is between $v_n(x)$ and $\hat{v}_n(x)$,  it follows that $c_n(x)$ is between $(p_n-1)v_n(x)^{p_n-2}$ and $(p_n-1)\hat{v}_n(x)^{p_n-2}$ for all $x\in\rn$. By \eqref{Ms}, passing to a subsequence, $|c_n|_{\infty}\leq 2$ for all $n\in\mathbb N$, $(p_n-1)v_n^{p_n-2}\to 1$, $(p_n-1)\hat{v}_n^{p_n-2}\to 1$ pointwisely in $\rn$, and 
\begin{align*}
c_n(x)\rightarrow 1\qquad \text{ pointwisely in \ }\rn.    
\end{align*}
We also know, by Theorem~\ref{thm:asym1}, that 
\begin{align}\label{alphalim}
\alpha_n\to \Lambda_1\qquad \text{ as }    n\to\infty.
\end{align}
Then $|g_n|_{\infty}\leq C$ and, by elliptic regularity, $w_n\to w$ locally uniformly in $\rn$, where $|w|_{\infty}=1$ and $w\neq 0$.
 Furthermore, testing \eqref{w_neq} with $w_n$,
 \begin{align}
\|w_n\|^2=\int_{\rn}Q_\Omega \alpha_n c_n w_n^2
\leq \int_{\Omega}\alpha_n c_n w_n^2
\leq C|\Omega|.
 \end{align}
Hence, up to a subsequence, $w_n$ converges to $w$ weakly in  $D^{1,2}(\rn)$ and in $L_{loc}^2(\rn)$. Letting $n\to\infty$ in the weak formulation of \eqref{w_neq}, we obtain that $w$ is a weak solution to
\begin{align*}
-\Delta w = Q_\Omega\Lambda_1(\Omega) w \quad \text{ in } \rn,\qquad 
w> 0\quad \text{ in } \rn.
\end{align*}
Then $w=\phi_1$. Now, observe that
\begin{align*}
0
&=\int_{\rn}(-\Delta)v_n \hat v_n - (-\Delta)\hat v_n v_n
=\int_{\rn}Q_\Omega \alpha_n v_n \hat v_n (v_n^{p_n-2} - \hat v_n^{p_n-2})\\
&=\int_{\rn}Q_\Omega \alpha_n v_n \hat v_n\left( \int_0^1 ( s v_n+(1-s)\hat v_n)^{p_n-3}\d s\right)(p_n-2) (v_n-\hat v_n).
\end{align*}
This implies that 
\begin{align*}
0
&=\int_{\rn}Q_\Omega v_n \hat v_n w_n d_n,\qquad\text{where \ } d_n:=\int_0^1 (s v_n+(1-s)\hat v_n)^{p_n-3}\d s.
\end{align*}
By \eqref{vs}, $s v_n+(1-s)\hat v_n \to c\phi_1$ pointwisely in $\rn.$ Moreover, by Lemma~\ref{decayzn:lem}, \eqref{vs2}, and \eqref{Ms},
\begin{align*}
|Q_\Omega v_n \hat v_n w_n d_n|
&\leq 
v_n \hat v_n \min\{v_n^{p_n-3},\hat v_n^{p_n-3}\}
=\min\{v_n^{p_n-2}\hat v_n, v_n \hat v_n^{p_n-2}\}\\
&\leq \max\{|v_n|_\infty^{p_n-2},|\hat v_n|_\infty^{p_n-2}\} \min\left\{|\hat v_n|_\infty \frac{\hat v_n}{|\hat v_n|_\infty},
|v_n|_\infty \frac{v_n}{|v_n|_\infty}\right\}\\
&\leq C \min\left\{\frac{\hat u_n}{|u_n|_\infty},
\frac{u_n}{|u_n|_\infty}\right\}
\leq C\min\{1,|x|^{-\frac{2}{q-2}}\}
\end{align*}
for some $C>0$, $n$ large enough, and $q>2$ so that $x\mapsto |x|^{-\frac{2}{q-2}}\in L^1(\rn\backslash B_1(0))$. Then, by dominated convergence, 
\begin{align*}
0&=\lim_{n\to\infty}\int_{\rn}Q_\Omega v_n \hat v_n w_n d_n
=c'\int_{\rn}Q_\Omega \phi_1^2>0
\end{align*}
for some $c'>0,$ which is a contradiction. 

\end{proof}

\begin{remark}\label{obs:rmk}
In the classical blow-up argument to show uniqueness of positive solutions, as in \cite{lin1994uniqueness}, the idea is to consider two different positive solutions $u_n$ and $\hat u_n$ (not necessarily of least energy type), to define $z_n:=\frac{u_n-\hat u_n}{|u_n-\hat u_n|_\infty}$ and to show that $z_n\to \phi_1$ locally.  This can be done also in the setting of problem \eqref{main}. However, the next steps of the arguments are problematic. One wishes to show that $z_n$ must change sign. This is usually done by contradiction and using integration by parts, but in our case we obtain that 
\begin{align*}
    0=\int_{\rn}(-\Delta)u_n \hat u_n-u_n (-\Delta)\hat u_n=\int_{\rn}Q_\Omega u_n \hat u_n (u_n^{p_n-2} - \hat u_n^{p_n-2}).
\end{align*}
Due to the weight $Q_\Omega$, this does not yield a direct contradiction and one cannot conclude as in the classical case with Dirichlet boundary conditions (and without the coefficient $Q_\Omega$). This is the main obstruction to generalize our uniqueness result to general positive solutions. 
\end{remark}

\section{Sharp decay below and at the Serrin exponent}\label{decay:sec}

Let $N\geq 3$ and let $p_s:=\frac{2N-2}{N-2}$ denote the Serrin exponent. Next we show Proposition~\ref{decayz}.

\begin{proof}[Proof of Proposition~\ref{decayz}]
\underline{Case $p\in(2,p_s)$:} Let $v(x)=C_p|x|^{-\frac{2}{p-2}}$ with $C_p=\left(\frac{4 p-2N(p-2)-4}
{(p-2)^2}\right)^{\frac{1}{p-2}}$. Direct calculations  show that $v$ is a classical solution to $-\Delta v = -v^{p-1}$ in $\r^N\backslash\{0\}.$ Let $z:=w-Cv$, with $C\in(0,1)$ such that $z>0$ on $\partial \Omega.$ Note that, 
\begin{align*}
-\Delta z 
= -w^{p-1}+Cv^{p-1}+((Cv)^{p-1}-(Cv)^{p-1})
= -w^{p-1}+(Cv)^{p-1}+(1-C^{p-2})Cv^{p-1}\quad \text{for $x\in \r^N\backslash \Omega$}.
\end{align*}
Let 
\begin{align}\label{cdef}
c(x):=
\begin{cases}
\frac{w^{p-1}(x) - (Cv(x))^{p-1}}{w(x)-Cv(x)}, &\text{if \ }w(x)\neq Cv(x),\\
0, &\text{if \ }w(x)=Cv(x).
\end{cases}
\end{align}
Then $c\geq 0$, and $z$ is a classical solution to
\begin{align*}
-\Delta z +cz=(1-C^{p-2})Cv^{p-1}\geq 0\quad \text{ in } \r^N\backslash \Omega,\qquad 
z>0\quad \text{ on }\partial \Omega.
\end{align*}
By the maximum principle, $z\geq 0$ in $\r^N\backslash \Omega$.  Indeed, assume by contradiction that $x_0\in \r^N\backslash \Omega$ is such $z(x_0)<0$ (note that $\lim_{|x|\to \infty}z(x)=0$). Then $-\Delta z(x_0)\leq 0$, $c(x_0)z(x_0)<0$ and, therefore,
\begin{align*}
0>-\Delta z(x_0) +c(x_0)z(x_0)\geq 0,
\end{align*}
which would yield a contradiction. Then, since $\partial\Omega$ is a compact set, by the definition of $z$, there exists $C'>0$ such that $w(x)\geq C'|x|^{-\frac{2}{p-2}}$ for all $|x|>1$.   

Now we obtain the bounds from above. Using the same notation, if $C>1$ is such that $z<0$ on $\partial \Omega.$ Then,
\begin{align*}
-\Delta z +cz=(1-C^{p-2})Cv^{p-1}\leq 0\quad \text{ in } \r^N\backslash \Omega,\qquad 
z<0\quad \text{ on }\partial \Omega.
\end{align*}
Arguing as before, $z\leq 0$ in $\r^N\backslash \Omega$, namely, there exists $C'>0$ such that $w(x)\leq C'|x|^{-\frac{2}{p-2}}$ for $|x|>1.$ This yields \eqref{bds}. 

\underline{Case $p=p_s$:}  This case can be argued similarly using a different comparison function. Let 
\begin{align*}
v(x)=k_N\left(|x|\sqrt{\ln(|x|)}\right)^{2-N},\qquad     k_N:=2^{2-N} \left(3 N^2-10 N+8\right)^{\frac{N-2}{2}}.
\end{align*}
Then, by direct calculations, 
\begin{align*}
 -\Delta v+v^{\frac{N}{N-2}}=c_N \frac{\ln (r)-1}{r^{N}(\ln(r))^{\frac{N+2}{2}}}>0\quad \text{ in }\r^N\backslash \overline{B_1(0)}
 \end{align*}
 with $c_N:=2^{-N} (N-2)^{N/2} N (3 N-4)^{\frac{N-2}{2}}.$ Let $z:=w-Cv$, with $C>1$ such that $z<0$ on $\partial B_R(0)$ with $R>1$ such that $\Omega\subset B_R(0).$ Then, 
\begin{align*}
-\Delta z 
< -w^{\frac{N}{N-2}}+Cv^{\frac{N}{N-2}}+((Cv)^{\frac{N}{N-2}}-(Cv)^{\frac{N}{N-2}})
= -w^{\frac{N}{N-2}}+(Cv)^{\frac{N}{N-2}}+(1-C^{\frac{2}{N-2}})Cv^{\frac{N}{N-2}}\quad \text{for $x\in \r^N\backslash B_R(0)$}.
\end{align*}
Let $c$ be as in \eqref{cdef}. Then $c\geq 0$ and $z$ is a classical solution to
\begin{align*}
-\Delta z +cz< (1-C^{\frac{2}{N-2}})Cv^{\frac{N}{N-2}}< 0\quad \text{ in } \r^N\backslash B_R(0),\qquad 
z<0\quad \text{ on }\partial B_R(0).
\end{align*}
As before, by the maximum principle, $z\leq 0$ in $\r^N\backslash B_R(0)$. Since $\partial\Omega$ is a compact set, by the definition of $z$, there exists $C'>0$ such that $w(x)\leq C'\left(|x|\sqrt{\ln(|x|)}\right)^{2-N}$ for all $|x|>1$.

Finally, we argue the bound from below in \eqref{bds3}. For this, we must adjust the constants. Let 
\begin{align*}
V(x)=K_N\left(|x|\sqrt{\ln(|x|)}\right)^{2-N},\qquad     K_N:=2^{\frac{2-N}{2}}(N-2)^{N-2}.
\end{align*}
Then, by direct calculations, 
\begin{align*}
 -\Delta V+V^{\frac{N}{N-2}}=C_N \frac{-1}{r^{N}(\ln(r))^{\frac{N+2}{2}}}<0\quad \text{ in }\r^N\backslash \overline{B_1(0)}
 \end{align*}
with $C_N:=2^{-\frac{N+2}{2}} (N-2)^{N-1} N>0$.  Let $Z:=w-\eps V$, with $\eps\in(0,1)$ such that $Z>0$ on $\partial B_R(0)$ with $R>1$ such that $\Omega\subset B_R(0).$ Then, 
\begin{align*}
-\Delta Z>-w^{\frac{N}{N-2}}+ \eps V^{\frac{N}{N-2}}+((\eps V)^{\frac{N}{N-2}}-(\eps V)^{\frac{N}{N-2}})=-w^{\frac{N}{N-2}}+(\eps V)^{\frac{N}{N-2}}+(1-\eps^{\frac{2}{N-2}})\eps V^{\frac{N}{N-2}}\quad \text{for $x\in \r^N\backslash B_R(0)$}.
\end{align*}
Hence,
\begin{align*}
-\Delta Z +cZ\geq (1-\eps^{\frac{2}{N-2}})\eps V^{\frac{N}{N-2}}>0\quad \text{ in } \r^N\backslash B_R(0),\qquad 
Z>0\quad \text{ on }\partial B_R(0).
\end{align*}
Arguing as before, the maximum principle yields that $Z\geq 0$ in $\r^N\backslash B_R(0)$ and, since $w>0$ in $\rn,$ we can find $\eps'>0$ such that $w(x)> \eps'\left(|x|\sqrt{\ln(|x|)}\right)^{2-N}$ for $|x|>1.$ 
\end{proof}

\section{Multiplicity of least-energy solutions}\label{mult:sec}

In this section we show a symmetry breaking and multiplicity result for least energy solutions.  Let us consider the equation
\begin{equation}\label{problema critico}
    \begin{cases}
        -\Delta u=Q_\o|u|^{\ex-2}u,\\
        u\in D^{1,2}(\rn).
    \end{cases}
\end{equation}
This problem is studied in \cite{cfs25}. The solutions of \eqref{problema critico} are the critical points of the functional $J_\ast:D^{1,2}(\rn)\to\r$ given by
\begin{align*}
    J_\ast(u):=\frac{1}{2}\|u\|^2-\frac{1}{\ex}\irn Q_\o(x)|u|^{\ex},\qquad v\in D^{1,2}(\rn).
\end{align*}
The nontrivial critical points of $J_\ast$ belong to the Nehari manifold
\begin{align*}
  \cN_\ast= \Big\lbrace u\in D^{1,2}(\rn)\setminus\{0\}\,:\, \|u\|^2=\irn Q_\o(x)|u|^{\ex}\Big\rbrace. 
\end{align*}


\begin{lemma}\label{proyecciones} For $p\in (2,\ex)$, let $u_p$ be a positive least energy solution of \eqref{main}. There exists $p_0\in(2,2^\ast)$ such that $\irn Q_\o(x)|u_p|^{2^\ast}>0$ for all $p\in(p_0,2^{\ast})$. Let $t_p\in(0,\infty)$ be such that $\tilde{u}_p:=t_pu_p\in\cN_\ast$, i.e.,
\begin{align}\label{tp}
    t_p=\left(\frac{\|u_p\|^2}{\irn Q_\o(x)|u_p|^{\ex}}\right)^{\frac{1}{\ex-2}}=\left(\frac{\irn Q_\o(x) |u_p|^p}{\irn Q_\o(x)|u_p|^{\ex}}\right)^{\frac{1}{\ex-2}}.
\end{align} 
   Then, 
    \begin{align*}
        \lim_{p\to\ex}t_p=1\qquad\text{and}\qquad \lim_{p\to\ex}J_\ast(\tilde{u}_p)=c_\ast:=\inf_{u\in \cN_\ast}J_\ast(u)=\frac{1}{N}S^\frac{N}{2},
    \end{align*}
    where $S$ is the best constant for the Sobolev embedding $D^{1,2}(\rn)\hookrightarrow L^{2^*}(\rn).$
\end{lemma}

\begin{proof}
By \cite[Proof of Proposition 2.3]{cfs25} we have that 
\begin{align}\label{igualdades}
\alpha_{2^*} = 
\inf_{\substack{u\in D^{1,2}(\mathbb{R}^N)\\ \int_{\rn}Q_\Omega|u|^{2^*}>0}}
\frac{\|u\|^2}
{\left( \displaystyle \int_{\rn}Q_\Omega|u|^{2^*} \right)^{2/2^*}}=
\inf_{\substack{u\in D_0^{1,2}(\Omega)\\ u\neq 0}}
\frac{\displaystyle \int_\Omega |\nabla u|^2}
{\left( \displaystyle \int_\Omega |u|^{2^*} \right)^{2/2^*}}
=\inf_{\substack{u\in D^{1,2}(\mathbb{R}^N)\\ u\neq 0}}
\frac{\|u\|^2}
{\left( \displaystyle \int_{\mathbb{R}^N} |u|^{2^*} \right)^{2/2^*}}.
\end{align}
Moreover, for $u\in D^{1,2}(\rn)$,
\begin{align}\label{des}
\left(\int_{\rn}Q_\Omega|u|^{p}\right)^\frac{1}{p}
\leq \left(\int_{\Omega}|u|^{p}\right)^\frac{1}{p}
\leq |\Omega|^\frac{2^*-p}{2^*p}\left(\int_{\Omega}|u|^{2^*}\right)^\frac{1}{2^*}
\leq |\Omega|^\frac{2^*-p}{2^*p}\left(\int_{\rn}|u|^{2^*}\right)^\frac{1}{2^*}.
\end{align}
Hence, 
\begin{align}
\alpha_{2^*}
&=\inf_{\substack{u\in D^{1,2}(\mathbb{R}^N)\\ u\neq 0}}
\frac{\|u\|^2}
{\left( \displaystyle \int_{\mathbb{R}^N} |u|^{2^*} \right)^{2/2^*}}
\leq \inf_{\substack{u\in D^{1,2}(\mathbb{R}^N)\\ \int_{\rn}Q_\Omega|u|^{p}>0}}
\frac{\|u\|^2}
{\left( \displaystyle \int_{\mathbb{R}^N} |u|^{2^*} \right)^{2/2^*}}\notag\\
&\leq |\Omega|^{\frac{2^*-p}{2^*p}}\inf_{\substack{u\in D^{1,2}(\mathbb{R}^N)\\ \int_{\rn}Q_\Omega|u|^{p}>0}}
\frac{\|u\|^2}
{\left( \displaystyle \int_{\rn}Q_\Omega|u|^{p} \right)^{2/p}}
=|\Omega|^{\frac{2^*-p}{2^*p}}\alpha_{p}.\label{igualdades2}
\end{align}
By Lemma~\ref{lem:bounds for alpha}, passing to the limit as $p\to \ex$,  we have that
\begin{align*}
    \alpha_{2^\ast}\leq\liminf_{p\to 2^\ast}\alpha_p\leq\limsup_{p\to 2^{\ast}}\alpha_p\leq a_1<\infty.
\end{align*}

Now, we claim that $\lim\limits_{p\to 2^*}\alpha_p = \alpha_{2^*}.$ Arguing by contradiction, assume that 
\begin{align}\label{chip}
    \alpha_{\ex}+\varepsilon<\limsup\limits_{p\to\ex}\alpha_p\qquad \text{ for some $\varepsilon>0$.}
\end{align} 
 From \eqref{igualdades} and the density of $C_c^\infty(\rn)$ in  $D^{1,2}(\rn)$, there exists $\tilde{u}\in C_c^\infty(\rn)$ such that
 \begin{align}\label{c1}
     \frac{\|\tilde{u}\|^2}{\left(\irn Q_\o |\tilde{u}|^{\ex}\right)^{2/\ex}}<\alpha_{ \ex}+\frac{\varepsilon}{2}.
 \end{align}
 Since the function $p\mapsto\irn Q_\o |\tilde{u}|^{p}$ is continuous in $(2,2^*]$, there exists $q\in(2,\ex)$ such that
 \begin{align}\label{c2}
     0<\irn Q_\o|\tilde{u}|^p
     \qquad\text{and}\qquad
     \left|\frac{\|\tilde{u}\|^2}{\left(\irn Q_\o |\tilde{u}|^{\ex}\right)^{2/\ex}}-\frac{\|\tilde{u}\|^2}{\left(\irn Q_\o |\tilde{u}|^{p}\right)^{2/p}}\right|<\frac{\varepsilon}{2}\qquad\text{ for all }\,p\in[q,\ex).
 \end{align}
 Then, by \eqref{c1}, \eqref{c2}, and \eqref{chip},
 \begin{align*}
     \alpha_p\leq \frac{\|\tilde{u}\|^2}{\left(\irn Q_\o |\tilde{u}|^{p}\right)^{2/p}}<\alpha_{\ex}+\varepsilon<\limsup_{p\to\ex}\alpha_p\qquad\text{ for all }\,p\in[q,\ex),
 \end{align*}
which yields a contradiction. Hence, 
\begin{align}\label{clim}
    \alpha_{\ex}=\lim_{p\to\ex}\alpha_p.
\end{align}

Let $u_p$ be a positive minimizer of \eqref{main}. It is easy to see that 
\begin{align}\label{n}
\frac{\|u_p\|^2}{\left(\irn Q_\o(x) |u_p|^{p}\right)^{2/p}}=\alpha_p.
\end{align}
Then, by \eqref{des},
\begin{align*}
    \alpha_{\ex}\leq\frac{\|u_p\|^2}{\left(\irn |u_p|^{\ex}\right)^{2/\ex}}\leq |\o|^{\frac{2(\ex-p)}{2*p}}\frac{\|u_p\|^2}{\left(\irn Q_\o(x) |u_p|^{p}\right)^{2/p}}= |\o|^{\frac{2(\ex-p)}{2*p}}\alpha_p.
\end{align*}
Passing to the limit as $p\to 2^*$ and using  \eqref{clim}, 
\begin{align}\label{concentracion}
    \lim_{p\to\ex}\irn|u_p|^{\ex}=\lim_{p\to\ex}\irn Q_\o|u_p|^{p}.
\end{align}
From \eqref{des}, \eqref{concentracion}, and \eqref{n} (using that $\|u_p\|^2=\irn Q_\o|u_p|^{p}$),
\begin{align}\label{limites}
  \lim_{p\to\ex}\io|u_p|^p=\lim_{p\to\ex}\io|u_p|^{\ex}=  \lim_{p\to\ex}\irn|u_p|^{\ex}=\lim_{p\to\ex}\irn Q_\o|u_p|^{p}=\lim_{p\to\ex}\alpha_p^{\frac{p}{p-2}}=\alpha_{\ex}^{\frac{N}{2}}.
\end{align}
This implies that $\lim_{p\to\ex}\int_{\rn\setminus\o}|u_p|^{\ex}=0$ and 
\begin{align}\label{qp}
  \lim_{p\to\ex}\irn Q_\o|u_p|^{\ex}=\alpha_{\ex}^{\frac{N}{2}}.  
\end{align}
Therefore, $\irn Q_\o|u_p|^{\ex}>0$ for $p$ close enough to $\ex$. The natural projection of $u_p$ to the Nehari manifold $\cN_{\ex}$ is given by $\widehat{u}_p=t_pu_p$ where $t_p$ is defined in \eqref{tp}. By \eqref{limites}
and \eqref{qp}, $\lim\limits_{p\to\ex}t_p=1$.

Finally, note that
\begin{align*}
    J_{\ast}(\widehat{u}_p)=\frac{\ex-2}{\ex 2}\|\widehat{u}_p\|^2=\frac{1}{N}t_p^2\|u_p\|^2=\frac{1}{N}t_p^2\alpha_{p}^{\frac{p}{p-2}}.
\end{align*}
Hence, $\lim_{p\to\ex}J_{\ast}(\widehat{u}_p)=\frac{1}{N}\alpha_{\ex}^{\frac{N}{2}}=c_{\ast},$ as claimed.  That $c_\ast=\frac{1}{N}S^\frac{N}{2}$ follows from \cite[Proposition 2.3 (c)]{cfs25}. 
\end{proof}

\begin{corollary}\label{casi burbuja}
    Let $p_k\in(2,\ex)$ be such that $p_k\to\ex$ and let $u_{p_k}$ be a positive minimizer to \eqref{main} with $p=p_k$. Then, after passing to a subsequence, there exist $\varepsilon_{k}>0$ and $\xi_k\in\o$ such that
    \begin{align*}
        \lim_{k\to\infty}\varepsilon_k^{-1}\text{dist}(\xi_k,\partial\o)=\infty\qquad\text{and}\qquad\lim_{k\to\infty}\|u_k-U_{\varepsilon_k,\xi_k}\|=0,
    \end{align*}
    where 
    \begin{align*}
    U_{\varepsilon_k,\xi_k}(x):=(N(N-2))^{\frac{N-2}{4}}\frac{\varepsilon_k^{\frac{N-2}{2}}}{\left(\varepsilon_k^2+|x-\xi_k|^{2}\right)^{\frac{N-2}{2}}}.
\end{align*}
\end{corollary}
\begin{proof}
Let $u_{p_k}$ be as in the statement and let $\widehat{u}_{p_k}=t_{p_k}u_{p_k}\in\cN_\ast$ with $t_{p_k}$ as in Lemma \ref{proyecciones}.   By \cite[Corollary 4.4]{cfs25} and Lemma~\ref{proyecciones}, there exist $\varepsilon>0$ and $\xi_k\in\o$ such that $\lim \limits_{k\to\infty}\varepsilon_k^{-1}\operatorname{dist}(\xi_k,\partial\o)=\infty$ and $\lim\limits_{k\to\infty}\|\widehat{u}_{p_k}-U_{\varepsilon_k,\xi_k}\|=0.$ Hence,
    \begin{align*}
        \|u_{p_k}-U_{\varepsilon_k,\xi_k}\|
        \leq\frac{1}{t_{p_k}}\Big(\|\tilde{u}_{p_k}-U_{\varepsilon_k,\xi_k}\|+|1-t_{p_k}|\|U_{\varepsilon_k,\xi_k}\|\Big)=o(1)\qquad \text{ as }k\to\infty,
    \end{align*}
    and the statement follows.
\end{proof}

We now use Corollary \ref{casi burbuja} to show the symmetry breaking result for least energy solutions of \eqref{main} for $p$ close to $\ex$ in annuli stated in the Introduction.  As a consequence, one can use rotations of this solution to obtain \emph{infinitely many} different least energy solutions of \eqref{main}. This contrasts with the case of $p$ close to 2, where, by Theorem \ref{exit p small}, the least energy solution is unique (and therefore radially symmetric whenever $\Omega$ is invariant under rotations).

\begin{proof}[Proof of Theorem \ref{thm:mult}]
Let $p_k\in(2,\ex)$ be such that $p_k\to\ex$ and let $u_{p_k}$ be a positive least energy solution of \eqref{main} with $p=p_k$. By Corollary~\ref{casi burbuja}, there are sequences $(\varepsilon_{k})\subset (0,\infty)$ and $(\xi_k)\subset \o$ such that, passing to a subsequence, 
\begin{align*}
u_{p_k}=U_{\varepsilon_k,\xi_k}+o(1)\quad \text{ in $D^{1,2}(\rn)$},
\qquad 
\eps_k\to 0,\qquad 
\xi_k\to \xi_0\in \overline{\Omega},
\end{align*}
  as $k\to \infty$. The result now follows since $U_{\varepsilon_k,\xi_k}$ is not invariant under rotations with respect to the origin.  The foliated Schwarz symmetry follows from \cite[Theorem 1.4]{CHS25}.
\end{proof}

\begin{remark}\label{rmk:table}

To close this paper, we include a table that highlights some differences among positive solutions of \eqref{main} as the exponent \(p\) varies.  

\medskip
\begin{center}
\renewcommand{\arraystretch}{1.3}
\begin{tabular}{|c|c|c|}
\hline
$\mathbf{p}$ & \textbf{Uniqueness vs multiplicity} & \textbf{Decay} \\
\hline
$p<2^*$ close to $2^*$ & Multiplicity of least energy solutions in some domains & $|u(x)|\leq C|x|^{2-N}$\\
\hline
$p>2$ close to $2$ & Uniqueness of least energy solutions & $|u(x)|\leq C|x|^{\frac{2}{2-p}}$\\
\hline
$p=2$ & First eigenfunction is unique up to a multiplicative factor & $|\phi_1(x)|\leq C |x|^{-\frac{N-1}{2}}e^{-\sqrt{\Lambda_1}|x|}$\\
\hline
\multirow{3}{*}{$p\in(1,2)$} & Uniqueness of least energy solutions & \multirow{3}{*}{$u$ has compact support} \\
\cline{2-2}
 & Uniqueness of positive solutions if $\Omega$ is a connected set & \\
\cline{2-2}
 & Uniqueness of positive solutions for $p\in(q_0(\Omega),2)$& \\
\cline{2-2}
 & Multiplicity of positive solutions in some disconnected domains & \\
\hline
\end{tabular}

\end{center}

\medskip

Here $q_0(\Omega)\in (1,2)$ is some exponent which depends on the geometry of $\Omega.$ Some of these results are proved in this paper, while others can be found in \cite{CHS25,CSS26}.

\end{remark}

\subsection*{Acknowledgments}
We thank Bendetta Pellaci for helpful discussions. We thank the anonymous referee for helpful comments and suggestions.  

\bibliographystyle{plain}

\bigskip

\begin{flushleft}
\textbf{Mónica Clapp}\\
Instituto de Matemáticas\\
Universidad Nacional Autónoma de México \\
Campus Juriquilla\\
Boulevard Juriquilla 3001\\
76230 Querétaro, Qro., Mexico\\
\texttt{monica.clapp@im.unam.mx}
\medskip

\textbf{Cristian Morales-Encinos} and \textbf{Alberto Saldaña}\\
Instituto de Matemáticas\\
Universidad Nacional Autónoma de México \\
Circuito Exterior, Ciudad Universitaria\\
04510 Coyoacán, Ciudad de México, Mexico\\
\texttt{cristianmorales@ciencias.unam.mx} \\
\texttt{alberto.saldana@im.unam.mx}
\medskip

\textbf{Mayra Soares}\\
Departamento de Matemática\\
Universidade de Brasília\\
Campus Darci Ribeiro\\
Asa Norte, Brasília\\ 70910-900, Brazil\\
\texttt{mayra.soares@unb.br}
\end{flushleft}

\end{document}